\documentclass{smfart} 

\usepackage[utf8]{inputenc} 
\usepackage[francais]{babel}
\usepackage{amssymb}
\usepackage{amsmath}
\usepackage{amsthm}
\usepackage{amssymb}
\usepackage[all]{xy}
\usepackage[T1]{fontenc}
\usepackage{geometry}
\usepackage{amsfonts}
\usepackage{mathabx}
\usepackage{mathrsfs}
\usepackage{enumitem}
\usepackage{hyperref}
\usepackage{smfthm}
\title[Comparaison des cycles proches par un morphisme sans pente]{théorème de comparaison pour les cycles proches par un morphisme sans pente}
\alttitle{Comparison theorem for nearby cycles of a morphism whitout slopes}

\author{Matthieu Kochersperger}

\address{UMR 7640 du CNRS, Centre de Mathématiques Laurent Schwartz, École polytechnique,
F–91128 Palaiseau cedex, France}

\email{matthieu.kochersperger@polytechnique.edu}



\begin{abstract}
Le but de cet article est de démontrer le théorème de comparaison entre les cycles proches algébriques et topologiques associés à un morphisme sans pente. Nous obtenons en particulier que dans le cas d'une famille de fonctions holomorphes sans pente, l'itération des isomorphismes de comparaison des cycles proches associés à chacune de ces fonctions ne dépend pas de l'ordre d'itération.
\end{abstract}

\begin{altabstract}
The goal of this article is to prove the comparison theorem between algebraic and topological nearby cycles of a morphism whitout slopes. We prove in particular that for a family of holomorphic functions whitout slopes, if we iterate comparison isomorphisms for nearby cycles of each function the result is independent of the order of iteration.
\end{altabstract}

\begin{document}

\date{}

\maketitle

\let\tilde\widetilde
\let\hat\widehat
\let\bar\overline

\theoremstyle{plain}
\newtheorem{theorem}{Théorème}[section]
\newtheorem{lemma}[theorem]{Lemme}
\newtheorem{proposition}[theorem]{Proposition}
\newtheorem{corollary}[theorem]{Corollaire}
\newtheorem{preuve}{\textit{Preuve} :}
\theoremstyle{definition}
\newtheorem{ccl}[theorem]{Conclusion}
\newtheorem{de}[theorem]{Définition}
\newtheorem{expl}[theorem]{Exemple}
\newtheorem{Rem}[theorem]{Remarque}

\def\restriction#1#2{\mathchoice
              {\setbox1\hbox{${\displaystyle #1}_{\scriptstyle #2}$}
              \restrictionaux{#1}{#2}}
              {\setbox1\hbox{${\textstyle #1}_{\scriptstyle #2}$}
              \restrictionaux{#1}{#2}}
              {\setbox1\hbox{${\scriptstyle #1}_{\scriptscriptstyle #2}$}
              \restrictionaux{#1}{#2}}
              {\setbox1\hbox{${\scriptscriptstyle #1}_{\scriptscriptstyle #2}$}
              \restrictionaux{#1}{#2}}}
\def\restrictionaux#1#2{{#1\,\smash{\vrule height .8\ht1 depth .85\dp1}}_{\,#2}}

\tableofcontents

\section{Introduction}

\subsection{Théorème de comparaison pour une fonction}

Soit $X$ une variété analytique complexe et $f:X\to\mathbb{C}$ une fonction holomorphe. Soit $(\mathcal{F},\mathcal{M})$ la donnée d'un faisceau pervers sur $X$ et d'un $\mathcal{D}_X$-module holonome régulier associés par la correspondance de Riemann-Hilbert, c'est-à-dire $\mathcal{F}=\mathbf{DR}_X(\mathcal{M})$. Le foncteur \emph{cycles proches topologiques} $\Psi_f$ de P. Deligne associe à $\mathcal{F}$ un faisceau pervers à support $f^{-1}(0)$ muni d'un automorphisme de monodromie. Prolongeant une construction de B. Malgrange \cite{M2}, M. Kashiwara   introduit dans \cite{Kashiwara83} le foncteur \emph{cycles proches algébriques} $\Psi_f^{\text{alg}}$ (voir aussi \cite{MM}) qui associe à $\mathcal{M}$ un $\mathcal{D}_X$-module holonome régulier à support $f^{-1}(0)$ muni d'un automorphisme de monodromie. Ces deux foncteurs sont reliés par un isomorphisme de comparaison qui commute à la monodromie:
\begin{equation}
\Psi_f(\mathcal{F}) \simeq \mathbf{DR}_{X}\Psi_{f}^{\text{alg}}(\mathcal{M}).\label{thcomp}\end{equation}

\subsection{Théorème de comparaison pour plusieurs fonctions}

Soit maintenant $p\geq 2$ et $f_1,...,f_p$ des fonctions holomorphes sur $X$. Notons $\boldsymbol{f}=(f_1,...,f_p):X\to\mathbb{C}^p$ le morphisme associé. En général, les foncteurs $\Psi_{f_i}$ $(i = 1, . . . , p)$ ne commutent pas
entre eux, de même que les foncteurs $\Psi^{\text{alg}}_{f_i}$.

Dans \cite{Maisonobe} Ph. Maisonobe montre que sous la condition \emph{sans pente} pour le couple $(\boldsymbol{f},\text{car}(\mathcal{F}))$ on peut définir les foncteurs cycles proches topologiques et algébriques associés à $\boldsymbol{f}$. Il montre alors l'existence d'isomorphismes
\[\Psi_{\boldsymbol{f}}\mathcal{F}\simeq \Psi_{f_{\sigma(1)}}...\Psi_{f_{\sigma(p)}}\mathcal{F}\]
et
\[\Psi^{\text{alg}}_{\boldsymbol{f}}\mathcal{M}\simeq \Psi^{\text{alg}}_{f_{\sigma(1)}}...\Psi^{\text{alg}}_{f_{\sigma(p)}}\mathcal{M}\]
pour toute permutation $\sigma$ de $\{1,...,p\}$.
Ceci assure la commutativité des foncteurs cycles proches associés aux fonctions $f_i$ pour $1\leq i\leq p$. Dans l'introduction Ph. Maisonobe mentionne que, par itération de l'isomorphisme \eqref{thcomp}, ses résultats permettent d'obtenir pour tout $\sigma$ des isomorphismes de comparaison
\begin{equation}\label{compité}
\Psi_{\boldsymbol{f}}\mathcal{F}\simeq\Psi_{f_{\sigma(1)}}...\Psi_{f_{\sigma(p)}}(\mathcal{F}) 
\simeq \mathbf{DR}_{X}\Psi_{f_{\sigma(1)}}^{\text{alg}}...\Psi_{f_{\sigma(p)}}^{\text{alg}}(\mathcal{M})\simeq \mathbf{DR}_{X}\Psi^{\text{alg}}_{\boldsymbol{f}}\mathcal{M}.
\end{equation}

Dans cet article, nous montrerons (corollaire \ref{thcomparaison2}) que cet isomorphisme ne dépend pas de la permutation $\sigma$. Pour ce faire, nous exhibons un morphisme de comparaison entre $\Psi_{\boldsymbol{f}}\mathcal{F}$ et $\mathbf{DR}_{X}\Psi^{\text{alg}}_{\boldsymbol{f}}\mathcal{M}$ et nous montrons qu'il coïncide avec les isomorphismes de comparaison itérés \eqref{compité} pour toute permutation $\sigma$.

\subsection{Un exemples de morphisme sans pente}

On appelle singularité quasi-ordinaire un germe d'espace analytique réduit admettant une projection finie sur $\mathbb{C}^p$ dont le lieu de ramification est contenu dans un diviseur à croisements normaux. Si $S$ est une hypersurface de $\mathbb{C}^n$ à singularité quasi-ordinaire définie par une fonction holomorphe $f$, il existe une projection $\pi:\mathbb{C}^n\to\mathbb{C}^{n-1}$ quasi-ordinaire pour $S$. Le faisceau $\Psi_f\underline{\mathbb{C}}_{\mathbb{C}^n}$ est pervers et dans cette situation le couple $(\pi,\text{car}(\Psi_f\underline{\mathbb{C}}_{\mathbb{C}^n}))$ est sans pente.

Les singularités quasi-ordinaires apparaissent en particulier dans la méthode de Jung de résolution des surfaces singulières (voir \cite{Lipman}).

\subsection*{Remerciements}

Cet article a été écrit dans le cadre de ma thèse sous la direction de Claude Sabbah que je remercie vivement pour ses nombreux conseils durant l'élaboration de ce travail. Je remercie Philippe Maisonobe pour l'intérêt qu'il a porté à ce travail. Je remercie également le rapporteur pour ses remarques constructives.

\section{$V$-multifiltration canonique et foncteurs cycles proches}

Dans cette section on définit les cycles proches algébriques à l'aide de la $V$-multifiltration canonique d'un $\mathcal{D}_X$-module sans pente. On démontre des propriétés de cette multifiltration ainsi que de ses gradués. On définit ensuite les cycles proches topologiques associés à plusieurs fonctions. Enfin on introduit les fonctions de classe de Nilsson à plusieurs variables et on en montre des propriétés utilisées dans la section suivante pour établir un lien entre cycles proches algébriques et cycles proches topologiques.

\subsection{$V$-multifiltration canonique d'un $\mathcal{D}_X$-module sans pente}

On notera dans la suite
\begin{itemize}

\item $d_x:=\dim_\mathbb{C} X$
\item $\partial_{i}  :=  \partial_{t_i}$
\item $E_i  :=  t_i\partial_{i}$
\item $\boldsymbol{x}:=(x_1,...,x_{d_X-p})$
\item $\boldsymbol{1}_i:=(0,...,0,1,0,...,0)$ où le $1$ est en position $i$.
\item $\boldsymbol{\alpha}:=(\alpha_1,...,\alpha_p)$
\item $\boldsymbol{\alpha}_{I}:=(\alpha_i)_{i\in I}$ pour $I\subset \{1,...,p\}$
\item $\boldsymbol{t}:=t_1...t_p$
\item $\boldsymbol{t}^{{\boldsymbol{s}}}:=t_1^{s_1}...t_p^{s_p}$
\item $\mathcal{D}_X[{\boldsymbol{s}}]:=\mathcal{D}_X[s_1,...,s_p]$
\item $\boldsymbol{H}=\{H_1,...,H_p\}$ où les $H_i$ sont des hypersurfaces lisses dont la réunion définit un diviseur à croisements normaux. On se place ici dans le cas où il existe localement des coordonnées $(\boldsymbol{x},t_1,...,t_p)$ telles que
\[\begin{array}{cccc}
     \boldsymbol{f}:& X & \to & \mathbb{C}^p\\
       & (\boldsymbol{x},t_1,...,t_p) & \mapsto & (t_1,...,t_p)
     \end{array}\] 
et $H_i=f_i^{-1}(0)$. 

\end{itemize}

\begin{de}

Notons, pour tout $1\leq i\leq p$, $\mathcal{I}_i$ l'idéal de l'hypersurface $H_i$ et $\mathcal{I}^{\boldsymbol{k}}:=\prod_{i=1}^p\mathcal{I}^{k_i}_i$. Pour tout $\boldsymbol{k}\in\mathbb{Z}^p$ et pour tout $x\in X$ on définit:
\[\left(V_{\boldsymbol{k}}\mathcal{D}_X
\right)_x:=\{P\in \mathcal{D}_{X,x} ~|~ \forall\boldsymbol{m}\in\mathbb{Z}^p, P(\mathcal{I}^{\boldsymbol{k}+\boldsymbol{m}}_x)\subset \mathcal{I}^{\boldsymbol{k}+\boldsymbol{m}}_x\},
\]
ceci permet de définir une filtration croissante de $\mathcal{D}_X$ indexée par $\mathbb{Z}^p$.

Soit $\mathcal{M}$ un $\mathcal{D}_X$-module cohérent. Une $V$-\emph{multifiltration} $U_\bullet\mathcal{M}$ de $\mathcal{M}$ est une filtration croissante indexée par $\mathbb{Z}^p$ satisfaisant à $V_{\boldsymbol{k}}\mathcal{D}_X \cdot U_{\boldsymbol{k'}}\mathcal{M}\subset U_{\boldsymbol{k}+\boldsymbol{k'}}\mathcal{M}$ pour tout $\boldsymbol{k}$ et $\boldsymbol{k'}$ dans $\mathbb{Z}^p$. Une telle $V$-multifiltration est \emph{bonne} si elle est engendrée localement par un nombre fini de sections $(m_j)_{j\in J}$, c'est-à-dire que pour tout $j\in J$ il existe $\boldsymbol{k}_j\in\mathbb{Z}^p$ tel que pour tout $\boldsymbol{k}\in\mathbb{Z}^p$
\[U_{\boldsymbol{k}}\mathcal{M}=\sum_{j\in J}V_{\boldsymbol{k}+\boldsymbol{k}_j}\mathcal{D}_X \cdot m_j.\]

\end{de}

Lorsque des inégalités entre nombres complexes apparaîtront, l'ordre considéré sera toujours l'ordre lexicographique sur $\mathbb{C}$, c'est-à-dire
\[x+iy\leq a+ib \iff x<a ~\text{ou}~ (x=a ~\text{et}~ y\leq b).\]

En suivant \cite{Maisonobe} on commence par donner les conditions pour qu'un couple $(\boldsymbol{H},\mathcal{M})$ soit sans pente puis on définit la $V$-multifiltration de Malgrange-Kashiwara.

\begin{de}\label{defmultispe}
Soit $\mathcal{M}$ un $\mathcal{D}_X$-module cohérent.
\begin{enumerate}
\item On dit que le couple $(\boldsymbol{H},\mathcal{M})$ est \emph{multispécialisable sans pente} si au voisinage de tout point de $X$, il existe une bonne $V$-multifiltration $U_\bullet(\mathcal{M})$ de $\mathcal{M}$ et des polynômes $b_i(s)\in\mathbb{C}[s]$ pour tout $1\leq i\leq p$ tels que pour tout ${\boldsymbol{k}}\in\mathbb{Z}^p$, $b_i(E_i+k_i)U_{{\boldsymbol{k}}}\mathcal{M}\subset U_{\boldsymbol{k-1}_i}\mathcal{M}$.
\item On dit que le couple $(\boldsymbol{H},\mathcal{M})$ est \emph{multispécialisable sans pente par section} si, pour toute section locale $m$ de $\mathcal{M}$, il existe des polynômes $b_i(s)\in\mathbb{C}[s]$ pour tout $1\leq i\leq p$ tels que $b_i(E_i)m\in V_{\boldsymbol{-1}_i}\mathcal{D}_X\cdot m$.
\end{enumerate}
\end{de}

Rappelons la proposition 1 de \cite{Maisonobe}:

\begin{proposition}\label{propmult1}
Les deux définitions précédentes sont équivalentes et si la première est satisfaite pour une bonne $V$-multifiltration de $\mathcal{M}$, elle l'est pour toute. On dit alors que le couple $(\boldsymbol{H},\mathcal{M})$ est \emph{sans pente}.
\end{proposition}

On fixe $\mathcal{M}$ un $\mathcal{D}_X$-module cohérent tel que le couple $(\boldsymbol{H},\mathcal{M})$ soit sans pente.

\begin{de}

Le polynôme unitaire de plus bas degré vérifiant la propriété 1. de la définition  pour l'indice $i$ est appelé \emph{polynôme de Bernstein-Sato d'indice $i$ de la $V$-multifiltration $U_\bullet(\mathcal{M})$}, on le note $b_{i,U_\bullet(\mathcal{M})}$.

Le polynôme unitaire de plus bas degré vérifiant la propriété 2. de la définition  pour l'indice $i$ est appelé \emph{polynôme de Bernstein-Sato d'indice $i$ de la section $m$}, on le note $b_{i,m}$.

\end{de}

\begin{proposition}\label{propmult2}
 
Soient, pour $1\leq i\leq p$, des sections $\sigma_i: \mathbb{C}/\mathbb{Z}\to \mathbb{C}$ de la projection naturelle $\mathbb{C}\to\mathbb{C}/\mathbb{Z}$. Il existe une unique bonne $V$-multifiltration $V^\sigma_\bullet(\mathcal{M})$ de $\mathcal{M}$ telle que pour tout $i$ les racines de $b_{i,V^\sigma_\bullet(\mathcal{M})}$ soient dans l'image de $\sigma_i$.

\end{proposition}

La démonstration de cette proposition et de la proposition \ref{propmult3} est identique à celle du théorème 1. de \cite{Maisonobe}.

\begin{de}

On définit la multifiltration $V_\bullet(\mathcal{M})$ indexée par $\mathbb{C}^p$ et vérifiant:
\[\forall x\in X,~ V_{\boldsymbol{\alpha}}(\mathcal{M})_x:=\{m\in\mathcal{M}_x;~ s_i\geq -\alpha_i-1,~\forall s_i\in b_{i,m}^{-1}(0) ~\textrm{et}~ 1\leq i\leq p\}.\]
Cette $V$-multifiltration est appelée \emph{$V$-multifiltration canonique} ou \emph{$V$-multifiltration de Malgrange-Kashiwara}.

Si on considère l'ordre partiel standard sur $\mathbb{C}^p$
\[\boldsymbol{\alpha}\leq \boldsymbol{\beta} \iff \alpha_i \leq \beta_i ~\text{pour tout}~ 1\leq i\leq p
\]
 on peut définir
\[V_{<\boldsymbol{\alpha}}(\mathcal{M}):=\sum_{\boldsymbol{\beta}<\boldsymbol{\alpha}}V_{\boldsymbol{\beta}}(\mathcal{M})\]
et
\[\textup{gr}_{\boldsymbol{\alpha}}(\mathcal{M}):=V_{\boldsymbol{\alpha}}(\mathcal{M})/V_{<\boldsymbol{\alpha}}(\mathcal{M}).\]
Soit $I\subset \{1,...,p\}$ et $I^c$ son complémentaire, on définit
\[V_{<\boldsymbol{\alpha}_I,\boldsymbol{\alpha}_{I^c}}(\mathcal{M}):=\sum_{\boldsymbol{\beta}_I<\boldsymbol{\alpha}_I}V_{\boldsymbol{\beta}_I,\boldsymbol{\alpha}_{I^c}}(\mathcal{M}).\]

\end{de}

\begin{proposition}\label{propmult3}

On a l'égalité des $V$-multifiltrations $V_{(<\boldsymbol{\alpha}_I,\boldsymbol{\alpha}_{I^c})+{\boldsymbol{k}}}(\mathcal{M})=V^{\sigma_{<\boldsymbol{\alpha}_I,\boldsymbol{\alpha}_{I^c}}}_{{\boldsymbol{k}}}(\mathcal{M})$ où $\sigma_{<\boldsymbol{\alpha}_I,\boldsymbol{\alpha}_{I^c}}$ est la section dont l'image est l'ensemble

\[\left\{\begin{array}{llcc}
\boldsymbol{a}\in\mathbb{C}^p & \textrm{tel que} & -\alpha_i-1\leq a_i<-\alpha_i ~\forall~ i\in I^c \\
  & et & -\alpha_i-1< a_i\leq-\alpha_i ~\forall~ i\in I
\end{array}\right\}.\]
Il existe un ensemble fini $A\subset [-1,0[^p$ tel que la $V$-multifiltration canonique soit indexée par $A+\mathbb{Z}^p$.
Ainsi la $V$-multifiltration canonique est cohérente.
\end{proposition}

Soit $I\subset\{1,...,p\}$ et $J\subset I^c$. Comme pour les $\mathcal{D}_X$-modules cohérents, on a une notion de $V_{\mathbf{0}_I}^{\mathbf{H}_I}\mathcal{D}_X$-module multispécialisable sans pente le long des hypersurfaces $\mathbf{H}_{J}:=(H_i)_{i\in J}$.

\begin{de}
Soit $\mathcal{M}$ un $V_{\mathbf{0}_I}^{\mathbf{H}_I}\mathcal{D}_X$-module cohérent et $J\subset I^c$, on note $q:=\# J$.
\begin{enumerate}
\item On dit que le couple $(\boldsymbol{H}_{J},\mathcal{M})$ est \emph{multispécialisable sans pente} (ou \emph{spécialisable} si $q=1$) si au voisinage de tout point de $X$, il existe une bonne $V$-multifiltration $U_\bullet(\mathcal{M})$ de $\mathcal{M}$ et des polynômes $b_i(s)\in\mathbb{C}[s]$ pour tout $i\in J$ tels que pour tout ${\boldsymbol{k}}\in\mathbb{Z}^{q}$, $b_i(E_i+k_i)U_{{\boldsymbol{k}}}\mathcal{M}\subset U_{\boldsymbol{k-1}_i}\mathcal{M}$.
\item On dit que le couple $(\boldsymbol{H}_{J},\mathcal{M})$ est \emph{multispécialisable sans pente par section} (ou \emph{spécialisable par section} si $q=1$) si, pour toute section locale $m$ de $\mathcal{M}$, il existe des polynômes $b_i(s)\in\mathbb{C}[s]$ pour tout $i\in J$ tels que $b_i(E_i)m\in V_{-\mathbf{1}_i}^{\mathbf{H}_{J}}(V_{\mathbf{0}_I}^{\mathbf{H}_I}\mathcal{D}_X)\cdot m=V_{\boldsymbol{-1}_i}\mathcal{D}_X\cdot m$.

\end{enumerate}
\end{de}

\begin{Rem}

Comme pour les $\mathcal{D}_X$-modules (proposition \ref{propmult1}) les deux définitions sont équivalentes et si elle sont satisfaites on dira que le couple $(\boldsymbol{H}_{J},\mathcal{M})$ est sans pente (ou spécialisable si $q=1$). Les analogues des propositions \ref{propmult2} et \ref{propmult3} sont vraies pour les $V_{\mathbf{0}_I}^{\mathbf{H}_I}\mathcal{D}_X$-modules sans pente. 

\end{Rem}

\begin{proposition}\label{compomult}

Soit $I\subset \{1,...,p\}$ et $\mathcal{M}$ un $\mathcal{D}_X$-module cohérent tel que le couple $(\boldsymbol{H},\mathcal{M})$ soit sans pente. Alors le couple $(\boldsymbol{H}_I,\mathcal{M})$ est sans pente et pour tout $\boldsymbol{\alpha}_{I}$ le couple $(\boldsymbol{H}_{I^c},V_{\boldsymbol{\alpha_I}}^{\mathbf{H}_{I}}\mathcal{M})$ est sans pente. De plus, pour $I,J\subset\{1,...,p\}$ disjoints, les $V$-multifiltrations de Malgrange-Kashiwara satisfont à:
\begin{equation}\label{VfiltVfilt}
V_{\boldsymbol{\alpha}_I,\boldsymbol{\alpha}_J}^{\boldsymbol{H}_I\cup\boldsymbol{H}_J}(\mathcal{M})=V_{\boldsymbol{\alpha}_I}^{\boldsymbol{H}_I}(\mathcal{M})\cap V_{\boldsymbol{\alpha}_J}^{\boldsymbol{H}_J}(\mathcal{M})=V_{\boldsymbol{\alpha}_I}^{\boldsymbol{H}_I}\left( V_{\boldsymbol{\alpha}_J}^{\boldsymbol{H}_J}(\mathcal{M})\right).\end{equation}

\end{proposition}

On a également l'analogue de \cite[corollaire 4.2-7]{MM}

\begin{proposition}\label{exactV}

Pour tout $\alpha\in\mathbb{C}$ et tout $j\in I^c$, l'application $\mathcal{M}\mapsto V_\alpha^{H_j}(\mathcal{M})$ définit un foncteur exact de la catégorie des $V_{\mathbf{0}_I}^{\mathbf{H}_I}\mathcal{D}_X$-modules spécialisables le long de $H_j$ vers la catégorie des $V_0^{H_j}(V_{\mathbf{0}_I}^{\mathbf{H}_I})\mathcal{D}_X$-modules.

\end{proposition}

 Sachant que la $V$-multifiltration canonique est indexée par $A+\mathbb{Z}^p$ avec $A\subset[-1,0[^p$ fini, quitte à renuméroter ces indices on peut la supposer indexée par $\mathbb{Z}^p$ et appliquer la définition \ref{Fcomp} de l'appendice \ref{appB} aux $V$-filtrations canoniques de $\mathcal{M}$. 
 
 La condition \emph{sans pente} s'interprète de manière naturelle comme une condition de compatibilité des $V$-filtrations relatives aux différentes hypersurfaces considérées.

\begin{proposition}\label{compatVfiltr}

Si le couple $(\mathbf{H},\mathcal{M})$ est sans pente alors les filtrations $V_\bullet^{H_1}(\mathcal{M}),...,V_\bullet^{H_p}(\mathcal{M})$ de $\mathcal{M}$ sont \emph{compatibles} au sens de la définition \ref{Fcomp}.

\end{proposition}

\begin{proof}

Soit $\boldsymbol{\alpha}<\boldsymbol{\beta}\in\mathbb{C}^p$ et notons $I_q:=\{1,...,q\}$. On va construire par récurrence sur l'entier $p$ le $p$-hypercomplexe $X_p$ correspondant à la compatibilité des sous-objets \[V^{H_1}_{{\alpha}_1}(V_{\boldsymbol{\beta}_{I_p}}^{\boldsymbol{H}_{I_p}}\mathcal{M}),...,V^{H_p}_{{\alpha}_p}(V_{\boldsymbol{\beta}_{I_p}}^{\boldsymbol{H}_{I_p}}\mathcal{M})\subseteq V_{\boldsymbol{\beta}_{I_p}}^{\boldsymbol{H}_{I_p}}\mathcal{M}.\]

 D'après la remarque \ref{remcompatibilité}, deux filtrations sont toujours compatibles. Supposons construit le $q$-hypercomplexe $X_q$. D'après la proposition \ref{compomult} la propriété sans pente assure que les objets qui apparaissent dans $X_q$ sont des $V_{\mathbf{0}_{I_q}}^{\boldsymbol{H}_{I_q}}\mathcal{D}_X$-modules cohérents spécialisables le long de $H_{q+1}$. On déduit alors de la proposition \ref{exactV} que l'application de $V^{H_{q+1}}_{\alpha_{q+1}}(.)$ et $V^{H_{q+1}}_{\beta_{q+1}}(.)$ à de tels objets sont deux foncteurs exacts munis d'un monomorphisme de foncteurs donné par l'inclusion naturelle déduite de l'inégalité $\alpha_{q+1}\leq \beta_{q+1}$. On applique alors ces deux foncteurs à $X_q$, la fonctorialité fournit un $(q+1)$-hypercomplexe 
\[\xymatrix{0 \ar[r] & V^{H_{q+1}}_{\alpha_{q+1}}(X^q) \ar@{^{(}->}[r]^{i} &  V^{H_{q+1}}_{\beta_{q+1}}(X^q) \ar[r] &  \text{Coker}(i) \ar[r] & 0. 
}\]
C'est le $(q+1)$-hypercomplexe $X_{q+1}$ voulu. L'exactitude des différentes suites courtes provient de l'exactitude des suite courtes de $X^q$, de l'exactitude des foncteurs $V^{H_{q+1}}$-filtration ainsi que de l'exactitude du foncteur Coker(.) appliqué à des inclusions (lemme du serpent). On utilise également ici les identifications \eqref{VfiltVfilt}. Ceci nous donne par récurrence le $p$-hypercomplexe $X_p$. En prenant alors la limite inductive des $p$-hypercomplexes $X_p$ sur $\boldsymbol{\beta}\in\mathbb{C}^p$ on obtient le $p$-hypercomplexe correspondant à la compatibilité des sous-objets
\[V^{H_1}_{{\alpha}_1}(\mathcal{M}),...,V^{H_p}_{{\alpha}_p}(\mathcal{M})\subseteq \mathcal{M}.\]
Ceci étant vérifié pour tout $\boldsymbol{\alpha}\in\mathbb{C}^p$ la proposition est démontrée.
\end{proof}

La proposition \ref{commutmultgrad} fournit le corollaire suivant

\begin{corollary}\label{commutmultgradV}

Si le couple $(\mathbf{H},\mathcal{M})$ est sans pente alors l'objet obtenu en appliquant successivement les gradués $\textup{gr}_{\alpha_{i}}^{H_{i}}$ par rapport aux $V$-filtrations canoniques $V^{H_i}_{\bullet}$ ne dépend pas de l'ordre dans lequel on applique ces foncteurs et est égal à $\textup{gr}_{\boldsymbol{\alpha}}(\mathcal{M})$.

\end{corollary}

\begin{proposition}

Soit $\mathcal{M}$ un $\mathcal{D}_X$-module cohérent tel que $(\boldsymbol{H},\mathcal{M})$ soit sans pente et soit $1\leq i\leq p$. Alors le $\mathcal{D}_X$-module $\mathcal{M}(*H_i)$ est cohérent et le couple $(\boldsymbol{H},\mathcal{M}(*H_i))$ est sans pente. De plus, pour tout $\boldsymbol{\alpha}$ vérifiant $\alpha_i<0$, le morphisme naturel de $V_{\boldsymbol{0}}\mathcal{D}_X$-modules:
\[V_{\boldsymbol{\alpha}}(\mathcal{M})\to V_{\boldsymbol{\alpha}}(\mathcal{M}(*H_i))\]
est un isomorphisme.

\end{proposition}

\begin{proof}

Comme $(\boldsymbol{H},\mathcal{M})$ est sans pente, $\mathcal{M}$ est spécialisable le long de $H_i$ et on peut appliquer \cite[proposition 4.4-3]{MM} qui assure que $\mathcal{M}(*H_i)$ est cohérent, spécialisable le long de $H_i$ et que pour $\alpha_i<0$,
\[V^{H_i}_{{\alpha}_i}(\mathcal{M})\to V^{H_i}_{{\alpha}_i}(\mathcal{M}(*H_i))\]
est un isomorphisme.

Montrons que le couple $(\boldsymbol{H},\mathcal{M}(*H_i))$ est sans pente. C'est un problème local, on peut supposer que $H_i=\{t_i=0\}$. Soit $m'$ une section de $\mathcal{M}(*H_i)$, on a $m'=m/{t_i}^{k}$ où $m$ est dans l'image de $\mathcal{M}\to\mathcal{M}[1/{t_i}]$ et $k\in\mathbb{N}$. Le couple $(\boldsymbol{H},\mathcal{M})$ étant sans pente, pour tout $1\leq j\leq p$ il existe un polynôme non nul $b_j(s_j)$ satisfaisant à 
\[b_j(E_j)m\in V_{-\boldsymbol{1}_j}(\mathcal{D}_X)m.\]
On a alors 
\begin{align*}
b_j(E_j)t_i^km' & \in V_{-\boldsymbol{1}_j}(\mathcal{D}_X)t_i^km' \\
t_i^k b_j(E_j+\delta_{ij}k)m' & \in t_i^kV_{-\boldsymbol{1}_j}(\mathcal{D}_X)m'.
\end{align*}
En divisant par $t_i^k$ on obtient, $b_j(E_j+\delta_{ij}k_i)m' \in V_{-\boldsymbol{1}_j}(\mathcal{D}_X)m'$, ce qui permet de conclure que $(\boldsymbol{H},\mathcal{M}(*H_i))$ est sans pente.

D'après la proposition \ref{compomult} $V^{H_i}_{{\alpha}_i}(\mathcal{M})$ et $ V^{H_i}_{{\alpha}_i}(\mathcal{M}(*H_i))$ sont des $V_{{0}}^{{H}_i}\mathcal{D}_X$-modules sans pente le long de $\boldsymbol{H}_{\{i\}^c}$ donc, si $\boldsymbol{\alpha}$ satisfait à $\alpha_i<0$, on a un isomorphisme
\[V_{\boldsymbol{\alpha}}(\mathcal{M}) \simeq V^{\boldsymbol{H}_{\{i\}^c}}_{\boldsymbol{\alpha}_{\{i\}^c}}\left(V^{H_i}_{\alpha_i}(\mathcal{M})\right)
\xrightarrow{\sim}
V^{\boldsymbol{H}_{\{i\}^c}}_{\boldsymbol{\alpha}_{\{i\}^c}}\left(V^{H_i}_{\alpha_i}(\mathcal{M}(*H_i))\right) \simeq
 V_{\boldsymbol{\alpha}}(\mathcal{M}(*H_i))\]
ce qui conclut la démonstration de la proposition.

\end{proof}

\begin{corollary}\label{loca}

Soit $\mathcal{M}$ un $\mathcal{D}_X$-module cohérent tel que $(\boldsymbol{H},\mathcal{M})$ soit sans pente. Alors le $\mathcal{D}_X$-module $\mathcal{M}(*(H_1\cup...\cup H_p))$ est cohérent et le couple $(\boldsymbol{H},\mathcal{M}(*(H_1\cup...\cup H_p)))$ est sans pente. De plus pour tout $\boldsymbol{\alpha}$ vérifiant $\alpha_i<0$ pour tout $1\leq i\leq p$, le morphisme naturel de $V_{\boldsymbol{0}}\mathcal{D}_X$-modules:
\[V_{\boldsymbol{\alpha}}(\mathcal{M})\to V_{\boldsymbol{\alpha}}(\mathcal{M}(*(H_1\cup...\cup H_p))\]
est un isomorphisme.

\end{corollary}

\begin{proof}

On effectue une récurrence sur le nombre d'hypersurfaces par rapport auxquelles on localise $\mathcal{M}$ et le corollaire est une conséquence immédiate de la proposition précédente.

\end{proof}

\subsection{Gradués d'un $\mathcal{D}_X$-module sans pente et cycles proches algébriques}

Ici on démontre des propriétés des gradués de la $V$-multifiltration de Malgrange-Kashiwara et on définit les cycles proches algébriques.

\begin{proposition}\label{nilp}

Soit $\mathcal{M}$ un $\mathcal{D}_X$-module tel que $(\boldsymbol{H},\mathcal{M})$ soit sans pente. Pour tout $\beta\in\mathbb{C}$ et tout $1\leq i\leq p$, l'endomorphisme $(E_i+\beta+1)$ de 
\[V_{\beta,\boldsymbol{\alpha}_{\{i\}^c}}(\mathcal{M})/V_{<\beta,\boldsymbol{\alpha}_{\{i\}^c}}(\mathcal{M})\]
est nilpotent.
\end{proposition}

\begin{proof}

Notons $\sigma:=\sigma_{\beta,\boldsymbol{\alpha}_{\{i\}^c}}$ et $b_i(s)$ le polynôme de Bernstein-Sato d'indice $i$ de la multifiltration correspondant à la section $\sigma$. Les racines de $b_i$ sont donc dans l'intervalle $[-\beta-1,- \beta[$. Soit $\ell$ la multiplicité de la racine $-\beta-1$ de $b_i$. On pose $b_i(s)=b'_i(s)(s+\beta+1)^\ell$. On considère comme dans la preuve de \cite[Théorème 1]{Kashiwara83} la $V$-multifiltration de $\mathcal{M}$ suivante:
\[U_{{\boldsymbol{k}}}(\mathcal{M}):=V^{\sigma}_{{\boldsymbol{k}}-\boldsymbol{1}_i}(\mathcal{M})+(E_i+k_i+\beta+1)^\ell V^{\sigma}_{{\boldsymbol{k}}}(\mathcal{M}).\]
On peut montrer que c'est une bonne $V$-multifiltration, que ses polynômes de Bernstein-Sato d'indice $j\neq i$ divisent ceux de $V^\sigma_\bullet$ et que son polynôme de Bernstein-Sato d'indice $i$ divise $b'(s)(s+\beta)^\ell$. Les racines de $b'(s)(s+\beta)^\ell$ sont dans $]-\beta-1,-\beta]$, par unicité la multifiltration $U_{\bullet}(\mathcal{M})$ est égale à la multifiltration $V^{\tilde{\sigma}}_\bullet(\mathcal{M})$ où $\tilde{\sigma}=\sigma_{<\beta,\boldsymbol{\alpha}_{\{ i\}^c}}$. On a donc $U_{\boldsymbol{0}}(\mathcal{M})=V_{<\beta,\boldsymbol{\alpha}_{\{i\}^c}}(\mathcal{M})$ et on en déduit que $(E_i+\beta+1)^\ell$ annule 
\[V_{\beta,\boldsymbol{\alpha}_{\{i\}^c}}(\mathcal{M})/V_{<\beta,\boldsymbol{\alpha}_{\{i\}^c}}(\mathcal{M}).\]

\end{proof}

Étant donnée la définition de $\textup{gr}_{\boldsymbol{\alpha}}(\mathcal{M})$ on déduit immédiatement de cette proposition le corollaire suivant:

\begin{corollary}\label{cornilp}

Soit $\mathcal{M}$ un $\mathcal{D}_X$-module tel que $(\boldsymbol{H},\mathcal{M})$ soit sans pente. Pour tout $\boldsymbol{\alpha}\in\mathbb{C}^p$ et tout $1\leq i\leq p$, l'endomorphisme $(E_i+\alpha_i+1)$ de $\textup{gr}_{\boldsymbol{\alpha}}(\mathcal{M})$ est nilpotent.

\end{corollary}

\begin{de}

Étant donné un couple $(\boldsymbol{H},\mathcal{M})$ sans pente, on définit les \emph{cycles proches algébriques} de $\mathcal{M}$ relatifs à la famille d'hypersurfaces $\boldsymbol{H}$ de la manière suivante
\[\Psi_{\boldsymbol{H}}\mathcal{M}:=
\displaystyle\bigoplus_{\boldsymbol{\alpha}\in[-1,0[^p}\textup{gr}_{\boldsymbol{\alpha}}(\mathcal{M}).
\]
C'est un $\textup{gr}_{\mathbf{0}}^{V}\mathcal{D}_X$-modules cohérent. Or, si l'on note $X_0:=\bigcap_{1\leq i\leq p}H_i$, on a $\textup{gr}_{\mathbf{0}}^{V}\mathcal{D}_X\simeq \mathcal{D}_{X_0}[E_1,...,E_p]$. Le corollaire \ref{cornilp} implique ainsi que $\Psi_{\boldsymbol{H}}\mathcal{M}$ est un $\mathcal{D}_{X_0}$-module cohérent. Les cycles proches algébriques sont munis d'endomorphismes de monodromie pour $1\leq i\leq p$
\[T_i:=\exp (-2i\pi E_i).\]

\end{de}

La proposition suivante est une conséquence du corollaire \ref{commutmultgradV}

\begin{proposition}\label{maisonalg}

Soit $I\subset \{1,...,p\}$, on a alors un morphisme naturel, fonctoriel en $\mathcal{M}$, de $\textup{gr}_{\mathbf{0}}^{V}\mathcal{D}_X$-modules
\[\Psi_{\mathbf{H}}\mathcal{M}\to \Psi_{\mathbf{H}_I}(\Psi_{\mathbf{H}_{I^c}}\mathcal{M})
\]
qui est un isomorphisme si le couple $(\boldsymbol{H},\mathcal{M})$ est sans pente.

\end{proposition}

Dans le cas général $\boldsymbol{f}:X\to \mathbb{C}^p$, l'inclusion du graphe de $\boldsymbol{f}$ permet de définir les cycles proches algébriques.

\begin{de}

Considérons le diagramme
\[\xymatrix{X \ar[r]^-{i_{\boldsymbol{f}}} \ar[rd]^{\boldsymbol{f}} & X\times \mathbb{C}^p \ar[d]^{\boldsymbol{\pi}=(\pi_1,...,\pi_p)} \\
 & \mathbb{C}^p.}\]
où $i_{\boldsymbol{f}}$ est le graphe de $\boldsymbol{f}$. Soit $H_i:=\pi_i^{-1}(0)$. D'après ce qui précède, si le couple $(\boldsymbol{H},{i_{\boldsymbol{f}}}_+\mathcal{M})$ est sans pente, alors $\Psi_{\boldsymbol{H}}{i_{\boldsymbol{f}}}_+\mathcal{M}$ est un $\mathcal{D}_{X\times 0}$-module cohérent à support $\{(\boldsymbol{x},0)|\boldsymbol{f}(\boldsymbol{x})=0\}$. On peut le voir comme un $\mathcal{D}_{X}$-module cohérent à support $\boldsymbol{f}^{-1}(0)$, on le note alors $\Psi_{\boldsymbol{f}}^{\text{alg}}\mathcal{M}$.

\end{de}

On déduit de la proposition \ref{maisonalg} l'isomorphisme 
\[\Psi_{\boldsymbol{f}}^{\text{alg}}\mathcal{M}\to \Psi_{\boldsymbol{f_I}}^{\text{alg}}(\Psi_{\boldsymbol{f_{I^c}}}^{\text{alg}}\mathcal{M}).
\]

\subsection{Cycles proches topologiques}

Ici on définit le foncteur cycles proches topologiques associé à une fonction $f:X \to \mathbb{C}^p$ et appliqué à la catégorie des complexes de faisceaux à cohomologie $\mathbb{C}$-constructible.

\begin{de}
Considérons le diagramme suivant où les carrés sont cartésiens:

\[\xymatrix{ f^{-1}(0) \ar[r]^i \ar[d] & X \ar[d]^f & X^* \ar[l]_j  \ar[d]^{\restriction{f}{X^*}} & \widetilde{X} \ar[l]_p \ar[d]^{\tilde{f}}\\
\{0\} \ar[r]^i & \mathbb{C}^p & (\mathbb{C}^*)^p \ar[l]_j & \widetilde{(\mathbb{C}^*)^p}. \ar[l]_p}\]
Ici $X^*=X-F^{-1}(0)$ avec $F=f_1...f_p$ et $\widetilde{(\mathbb{C}^*)^p}$ est le revêtement universel de $(\mathbb{C}^*)^p$. 

Si $\mathcal{F}$ est un complexe de faisceaux à cohomologie $\mathbb{C}$-constructible, on définit:
\[\Psi_f\mathcal{F}:=i^{-1}\boldsymbol{R}j_*p_*p^{-1}j^{-1}\mathcal{F}\]
c'est le \emph{foncteur cycles proches}. On peut identifier le morphisme $\widetilde{(\mathbb{C}^*)^p} \to (\mathbb{C}^*)^p$ à

\[\begin{array}{cccc}
     \exp :& \mathbb{C}^p & \to & (\mathbb{C}^*)^p\\
       & (z_1,...,z_p) & \mapsto & (e^{2i\pi z_1},...,e^{2i\pi z_p}).
     \end{array}\] 
Pour $1\leq i\leq p$ les translations 
\[\begin{array}{cccc}
     \tau_i :& \widetilde{(\mathbb{C}^*)^p} & \to & \widetilde{(\mathbb{C}^*)^p} \\
       & (z_1,...,z_i,...,z_p) & \mapsto & (z_1,...,z_i+1,...,z_p).
     \end{array}\] 
permettent d'induire des endomorphismes de monodromie $T_i:\Psi_f\mathcal{F}\to \Psi_f\mathcal{F}$.      
     
\end{de}

Supposons que les $f_i$ définissent un diviseur à croisements normaux $\mathbf{H}$ où $H_i=\{f_i=0\}$ et que $\mathcal{F}=\mathbf{DR}(\mathcal{M})$. Dans \cite{Maisonobe} Ph. Maisonobe démontre la proposition suivante

\begin{proposition}\label{maisontopo}

Soit $I\subset \{1,...,p\}$, il existe un morphisme naturel 
\begin{equation}
\Psi_f\mathcal{F}\to \Psi_{f_I}(\Psi_{f_{I^c}}\mathcal{F}).
\label{coupeproche}\end{equation}
De plus si le couple $(\mathbf{H},\mathcal{M})$ est sans pente alors ce morphisme est un isomorphisme.

\end{proposition}

\subsection{Fonctions de classe de Nilsson}

On se place ici dans le cas d'une famille d'hypersurfaces qui forment un diviseur à croisements normaux, quitte à diminuer $X$, on suppose qu'il existe un système de coordonnées $(\boldsymbol{x},t_1,...,t_p)$ tel que pour tout $1\leq i\leq p$, l'hypersurface $H_i$ ait pour équation $t_i=0$. On note 
 \[\begin{array}{cccc}
     \pi:& X & \to & \mathbb{C}^p\\
       & (\boldsymbol{x},t_1,...,t_p) & \mapsto & (t_1,...,t_p).
     \end{array}\] 
     
\begin{de}\label{Nilsson}

 Soit $\boldsymbol{\alpha}\in {[-1,0[}^p$ et ${\boldsymbol{k}}\in \mathbb{N}^p$. On note $\mathcal{N}_{\boldsymbol{\alpha},{\boldsymbol{k}}}$ la connexion méromorphe sur $\mathbb{C}^p$:
 \[\mathcal{N}_{\boldsymbol{\alpha},{\boldsymbol{k}}}=
 \bigoplus_{0 \leq \boldsymbol{\ell}\leq {\boldsymbol{k}}}\mathcal{O}_{\mathbb{C}^p}[\frac{1}{z_1...z_p}]e_{\boldsymbol{\alpha},\boldsymbol{\ell}}\]
avec la structure de $\mathcal{D}$-module donnée par la formule
\[z_i\partial_{z_i}e_{\boldsymbol{\alpha},\boldsymbol{\ell}}  = (\alpha_i+1)e_{\boldsymbol{\alpha},\boldsymbol{\ell}}+e_{\boldsymbol{\alpha},\boldsymbol{\ell}-\boldsymbol{1}_i}.\]
On définit $T_i$ le morphisme de monodromie d'indice $i$ par la formule 

\[T_ie_{\boldsymbol{\alpha},\boldsymbol{\ell}}  = \exp(2i\pi(\alpha_i+1))\sum_{0\leq m\leq \ell_i}\frac{(2i\pi)^m}{m!}e_{\boldsymbol{\alpha},\boldsymbol{\ell}-m.\boldsymbol{1}_i}.
\]
 
\end{de}

\begin{Rem}

Pour se souvenir de la structure de $\mathcal{D}$-module et de la monodromie il faut remarquer que la section $e_{\boldsymbol{\alpha},\boldsymbol{\ell}}$ se comporte comme la fonction multiforme $\boldsymbol{z}^{\boldsymbol{\alpha+1}}\frac{\log^{\ell_1}z_1}{\ell_1!}...\frac{\log^{\ell_p}z_p}{\ell_p!}$.

\end{Rem}

\begin{de}
 
Soit $\mathcal{M}$ un $\mathcal{D}_X$-module tel que le couple $(\boldsymbol{H},\mathcal{M})$ soit sans pente. On définit: 
\[\mathcal{M}_{\boldsymbol{\alpha},{\boldsymbol{k}}}=\mathcal{M}\otimes_{\pi^{-1}\mathcal{O}_{\mathbb{C}^p}}
\pi^{-1}\left(\mathcal{N}_{\boldsymbol{\alpha},{\boldsymbol{k}}}\right)=
\mathcal{M}[\frac{1}{t_1...t_p}]\otimes_{\pi^{-1}\mathcal{O}_{\mathbb{C}^p}}\pi^{-1}\left(\mathcal{N}_{\boldsymbol{\alpha},{\boldsymbol{k}}}\right).\]
D'autre part on a 
\[\mathcal{M}_{\boldsymbol{\alpha},{\boldsymbol{k}}}=\mathcal{M}\otimes_{\mathcal{O}_{X}}\pi^{+}\left(\mathcal{N}_{\boldsymbol{\alpha},{\boldsymbol{k}}}\right)\]
où $\pi^+$ est l'image inverse dans la catégorie des $\mathcal{D}$-modules. Ceci permet de munir $\mathcal{M}_{\boldsymbol{\alpha},{\boldsymbol{k}}}$ 
d'une structure naturelle de $\mathcal{D}_X$-module. Notons $Y:=\displaystyle\bigcap_{1\leq i\leq p} H_i$. La restriction de $\mathcal{M}_{\boldsymbol{\alpha},{\boldsymbol{k}}}$ à $Y$ est munie 
d'endomorphismes $T_i$ induits par les morphismes de monodromie de $\mathcal{N}_{\boldsymbol{\alpha},{\boldsymbol{k}}}$ et définis par:
\[T_i(m\otimes e_{\boldsymbol{\alpha},\boldsymbol{\ell}})=m\otimes T_ie_{\boldsymbol{\alpha},\boldsymbol{\ell}}.\]
     
\end{de}

\begin{proposition}\label{VNilsson}
 
Soit $\boldsymbol{\alpha}\in {[-1,0[}^p$ et ${\boldsymbol{k}}\in \mathbb{N}^p$ et $\mathcal{M}$ un $\mathcal{D}_X$-module tel que le couple $(\boldsymbol{H},\mathcal{M})$ soit sans pente.
Alors le couple $(\boldsymbol{H},\mathcal{M}_{\boldsymbol{\alpha},{\boldsymbol{k}}})$ est sans pente. De plus, pour tout $\boldsymbol{\beta}\in\mathbb{C}^p$, on a:
\[V_{\boldsymbol{\beta}}(\mathcal{M}_{\boldsymbol{\alpha},{\boldsymbol{k}}})=\bigoplus_{\boldsymbol{0}\leq\boldsymbol{\ell}\leq{\boldsymbol{k}}}V_{\boldsymbol{\alpha}+\boldsymbol{\beta}+\boldsymbol{1}}
 \left(\mathcal{M}[\frac{1}{t_1...t_p}]\right)e_{\boldsymbol{\alpha},\boldsymbol{\ell}}.
\]

\end{proposition}

On commence par un lemme qui sera utile dans la démonstration de cette proposition.

\begin{de}

Soit $(\boldsymbol{x},t_1,...,t_p)$ un système de coordonnées locales où $t_i=0$ est une équation de $H_i$. Soit $\mathcal{M}[1/\boldsymbol{t},{\boldsymbol{s}}]\boldsymbol{t}^{\boldsymbol{s}}$ le $\mathcal{O}_X[{\boldsymbol{s}}]$-module isomorphe à $\mathcal{M}[1/\boldsymbol{t},{\boldsymbol{s}}]$ par l'application $m\mapsto m\boldsymbol{t}^{\boldsymbol{s}}$. Il est muni d'une structure naturelle de $\mathcal{D}_X[{\boldsymbol{s}}]$-module par la formule:
\[\partial_i(m\boldsymbol{t}^{\boldsymbol{s}}):=(\partial_im)\boldsymbol{t}^{\boldsymbol{s}}+(\frac{s_im}{t_i})\boldsymbol{t}^{\boldsymbol{s}} \]

\end{de}

\begin{lemma}\label{equiv}
Soit $m$ une section locale de $\mathcal{M}[1/\boldsymbol{t}]$ et $b(s)\in\mathbb{C}[s]$. Les conditions suivantes sont équivalentes:

\begin{enumerate}

\item $b(E_i)m \in V_{\boldsymbol{-1}_i}(\mathcal{D}_X)m$ 
\item $b(-s_i-1)m\boldsymbol{t}^{\boldsymbol{s}}\in \mathcal{D}_X[{\boldsymbol{s}}]t_im\boldsymbol{t}^{\boldsymbol{s}}$

\end{enumerate}

\end{lemma}

\begin{proof}

Montrons que \textit{1} implique \textit{2}. Dans $\mathcal{M}[1/\boldsymbol{t},{\boldsymbol{s}}]\boldsymbol{t}^{\boldsymbol{s}}$ on a l'égalité
\[(t_i\partial_im)\boldsymbol{t}^{\boldsymbol{s}}=(-s_i-1)m\boldsymbol{t}^{\boldsymbol{s}}+\partial_i(t_im\boldsymbol{t}^{\boldsymbol{s}}).\]
On montre alors par récurrence que pour tout $k$
\[((t_i\partial_i)^km)\boldsymbol{t}^{\boldsymbol{s}}-(-s_i-1)^km\boldsymbol{t}^{\boldsymbol{s}}\in \mathcal{D}_X[{\boldsymbol{s}}]t_im\boldsymbol{t}^{\boldsymbol{s}}.\]
On a donc pour tout polynôme $b(s)\in\mathbb{C}[s]$
\[(b(E_i)m)\boldsymbol{t}^{\boldsymbol{s}}-b(-s_i-1)m\boldsymbol{t}^{\boldsymbol{s}}\in \mathcal{D}_X[{\boldsymbol{s}}]t_im\boldsymbol{t}^{\boldsymbol{s}}.\]
D'autre part, si $b(E_i)m \in V_{\boldsymbol{-1}_i}(\mathcal{D}_X)m$ une récurrence permet de montrer que $(b(E_i)m)\boldsymbol{t}^{\boldsymbol{s}}\in\mathcal{D}_X[{\boldsymbol{s}}]
t_im\boldsymbol{t}^{\boldsymbol{s}}$ et on en déduit \textit{2}.

Montrons que \textit{2} implique \textit{1}. D'une part, on peut montrer par récurrence que pour tout $k\in\mathbb{N}$ et tout $1\leq \ell\leq k$, il existe $m_{k,\ell}\in\mathcal{M}[1/\boldsymbol{t}]$ satisfaisant à:
\begin{equation}\label{sm}
s_i^km\boldsymbol{t}^{\boldsymbol{s}}=((-\partial_it_i)^km)\boldsymbol{t}^{\boldsymbol{s}}+
\sum_{\ell=1}^k\partial_i^\ell(m_{k,\ell}\boldsymbol{t}^{\boldsymbol{s}}).
\end{equation}
D'autre part, en faisant opérer les $\partial^{\boldsymbol{\alpha}}=\partial_1^{\alpha_1}...\partial_p^{\alpha_p}$ et en annulant les coefficients du polynôme en les $s_i$ que l'on obtient, on peut montrer le résultat suivant:
\begin{equation}\label{sumalpha}
\left[\sum_{\boldsymbol{\alpha}}\partial^{\boldsymbol{\alpha}}(m_{\boldsymbol{\alpha}}\boldsymbol{t}^{\boldsymbol{s}})=0\right] \Rightarrow \left[m_{\boldsymbol{\alpha}}=0 ~~ \forall  \boldsymbol{\alpha}\right]
\end{equation}
pour une somme finie sur les $\boldsymbol{\alpha}$. Enfin, si l'on regarde plus précisément la récurrence faite dans la première partie de la démonstration on obtient
\[(b(E_i)m)\boldsymbol{t}^{\boldsymbol{s}}-b(-s_i-1)m\boldsymbol{t}^{\boldsymbol{s}}\in \partial_i\mathcal{D}_X[{\boldsymbol{s}}]t_im\boldsymbol{t}^{\boldsymbol{s}}.\]
L'hypothèse \textit{2} implique
\[b(-s_i-1)m\boldsymbol{t}^{\boldsymbol{s}}=\sum_{\boldsymbol{\alpha},{\boldsymbol{k}}}\partial^{\boldsymbol{\alpha}}{\boldsymbol{s}}^{\boldsymbol{k}}A_{\boldsymbol{\alpha},{\boldsymbol{k}}}t_im\boldsymbol{t}^{\boldsymbol{s}}\]
où $A_{\boldsymbol{\alpha},{\boldsymbol{k}}}$ est un opérateur différentiel indépendant des $\partial_i$ pour tout $1\leq i\leq p$. En utilisant l'égalité \eqref{sm} on peut substituer les $s_j$ et on obtient 
\[(b(E_i)m)\boldsymbol{t}^{\boldsymbol{s}}-\sum_{\boldsymbol{k}}\left[(-t_1\partial_1-1)^{k_1}...(-t_p\partial_p-1)^{k_p}A_{\boldsymbol{0},{\boldsymbol{k}}}t_im\right]\boldsymbol{t}^{\boldsymbol{s}}=\sum_{\boldsymbol{\alpha}>\boldsymbol{0}}\partial^{\boldsymbol{\alpha}}(m_{\boldsymbol{\alpha}}\boldsymbol{t}^{\boldsymbol{s}})\]
avec $m_{\boldsymbol{\alpha}}\in\mathcal{M}[1/\boldsymbol{t}]$. En utilisant \eqref{sumalpha} et le fait que $(-t_1\partial_1-1)^{k_1}...(-t_p\partial_p-1)^{k_p}A_{\boldsymbol{0},\boldsymbol{k}}t_i\in V_{\boldsymbol{-1}_i}(\mathcal{D}_X)$ on conclut que $b(E_i)m \in V_{\boldsymbol{-1}_i}(\mathcal{D}_X)m$.

\end{proof}

\begin{proof}[Démonstration de la proposition \ref{VNilsson}]
 
On commence par montrer que le couple $(\boldsymbol{H},\mathcal{M}_{\boldsymbol{\alpha},{\boldsymbol{k}}})$ est sans pente. Quelque soit $1\leq i\leq p$, le $\mathcal{D}_{\mathbb{C}^p}$-module 
$\mathcal{N}_{\boldsymbol{\alpha},{\boldsymbol{k}}}/\mathcal{N}_{\boldsymbol{\alpha},{\boldsymbol{k}}-\boldsymbol{1}_i}$ s'identifie à $\mathcal{N}_{\boldsymbol{\alpha},{\boldsymbol{k}}-k_i.\boldsymbol{1}_i}$ On a donc la suite exacte:
\[0 \to \mathcal{N}_{\boldsymbol{\alpha},{\boldsymbol{k}}-\boldsymbol{1}_i} \to \mathcal{N}_{\boldsymbol{\alpha},{\boldsymbol{k}}} \to  \mathcal{N}_{\boldsymbol{\alpha},{\boldsymbol{k}}-k_i.\boldsymbol{1}_i} \to 0.\]
Pour tout ${\boldsymbol{k}}\in \mathbb{N}^p$ le $\pi^{-1}\mathcal{O}_{\mathbb{C}^p}$-module $\pi^{-1}\mathcal{N}_{\boldsymbol{\alpha},{\boldsymbol{k}}}$ est à fibres plates car libres,
il est donc acyclique pour le foncteur de produit tensoriel par $\mathcal{M}[\frac{1}{t_1...t_p}]$ et on a la suite exacte:
\begin{equation}\label{exactm}
0 \to \mathcal{M}_{\boldsymbol{\alpha},{\boldsymbol{k}}-\boldsymbol{1}_i} \to \mathcal{M}_{\boldsymbol{\alpha},{\boldsymbol{k}}} \to  \mathcal{M}_{\boldsymbol{\alpha},{\boldsymbol{k}}-k_i.\boldsymbol{1}_i} \to 0.
\end{equation}
Le module central est sans pente si et seulement si les deux autres modules le sont. En effet, comme dans le cas des bonnes $V$-filtration pour $p=1$ (cf \cite{MM}), une bonne $V$-multifiltration du terme central induit des bonnes $V$-multifiltration des termes extrêmes. On considère alors la suite exacte
  
\[0 \to U_{\boldsymbol{\ell}}\mathcal{M}_{\boldsymbol{\alpha},{\boldsymbol{k}}-\boldsymbol{1}_i} \to U_{\boldsymbol{\ell}}\mathcal{M}_{\boldsymbol{\alpha},{\boldsymbol{k}}} \to  U_{\boldsymbol{\ell}}\mathcal{M}_{\boldsymbol{\alpha},{\boldsymbol{k}}-k_i.\boldsymbol{1}_i} \to 0\]
et on observe que la condition multispécialisable sans pente de la définition \ref{defmultispe} est satisfaite pour le module central si et seuleument si elle l'est pour les deux autres modules. Par récurrence on est alors ramené à montrer que 
$(\boldsymbol{H},\mathcal{M}_{\boldsymbol{\alpha},\boldsymbol{0}})$ est sans pente.
Soit $m$ une section locale de $\mathcal{M}[\frac{1}{t_1...t_p}]$. D'après la proposition \ref{loca} le couple $(\boldsymbol{H},\mathcal{M}[\frac{1}{t_1...t_p}])$ est sans pente et par conséquent le lemme \ref{equiv} fournit localement, pour $1\leq i\leq p$, des polynômes $b_i$ non nuls vérifiant:
\[
b_i(s_i)m\boldsymbol{t}^{\boldsymbol{s}}\in \mathcal{D}_X[{\boldsymbol{s}}]t_im\boldsymbol{t}^{\boldsymbol{s}}.
\]
Par définition du $\mathcal{D}_X[s]$-module $\mathcal{M}[1/\boldsymbol{t},{\boldsymbol{s}}]\boldsymbol{t}^{\boldsymbol{s}}$, on obtient les équations:
\begin{equation}\label{bleu}
 b_i(s_i+\alpha_i+1)(m\otimes e_{\boldsymbol{\alpha},\boldsymbol{0}})\boldsymbol{t}^{\boldsymbol{s}}\in \mathcal{D}_X[{\boldsymbol{s}}]t_i(m\otimes e_{\boldsymbol{\alpha},\boldsymbol{0}})\boldsymbol{t}^{\boldsymbol{s}}.
\end{equation}

Soit ${\boldsymbol{k}}_0\in \mathbb{N}^p$ tel que pour tout $k_i\in \mathbb{N}$ vérifiant $k_i\geq k_{0,i}+1$, l'entier $-k_i$ n'est pas racine de $b_i(s_i+\alpha_i+1)\in\mathbb{C}[s_i]$. En remplaçant les $s_i$ par les entiers $k_i$ dans la relation \eqref{bleu} et en multipliant éventuellement par des $t_i$ on obtient que pour tout ${\boldsymbol{k}}\in \mathbb{Z}^p$
\[(m\otimes e_{\boldsymbol{\alpha},\boldsymbol{0}})\boldsymbol{t}^{\boldsymbol{k}}\in \mathcal{D}_X((m\otimes e_{\boldsymbol{\alpha},\boldsymbol{0}})\boldsymbol{t}^{-{\boldsymbol{k}}_0}).\]
De plus pour tout $1\leq i\leq p$, l'égalité $(\partial_i(m\otimes e_{\boldsymbol{\alpha},\boldsymbol{0}}))\boldsymbol{t}^{\boldsymbol{k}}=
\partial_i((m\otimes e_{\boldsymbol{\alpha},\boldsymbol{0}})\boldsymbol{t}^{\boldsymbol{k}})
+k_i(m\otimes e_{\boldsymbol{\alpha},\boldsymbol{0}})\boldsymbol{t}^{{\boldsymbol{k}}-\boldsymbol{1}_i}$ montre que $(\partial_i(m\otimes e_{\boldsymbol{\alpha},\boldsymbol{0}}))\boldsymbol{t}^{\boldsymbol{k}}\in \mathcal{D}_X((m\otimes e_{\boldsymbol{\alpha},\boldsymbol{0}})\boldsymbol{t}^{-{\boldsymbol{k}}_0})$ pour tout ${\boldsymbol{k}}\in \mathbb{Z}^p$. Comme $\mathcal{M}$ est engendré par un nombre fini de sections, en utilisant des extensions successives on peut supposer que $m$ engendre $\mathcal{M}$. On a donc $\mathcal{M}_{\boldsymbol{\alpha},\boldsymbol{0}}=\mathcal{D}_X((m\otimes e_{\boldsymbol{\alpha},\boldsymbol{0}})\boldsymbol{t}^{-{\boldsymbol{k}}_0})$. La filtration $\mathcal{D}_X(l)((m\otimes e_{\boldsymbol{\alpha},\boldsymbol{0}})\boldsymbol{t}^{-{\boldsymbol{k}}_0}))$ étant une bonne filtration du $\mathcal{D}_X$-module $\mathcal{M}_{\boldsymbol{\alpha},\boldsymbol{0}}$, celui-ci est cohérent. Les équations \eqref{bleu} ainsi que le lemme \ref{equiv} permettent alors de conclure que $(\boldsymbol{H},\mathcal{M}_{\boldsymbol{\alpha},\boldsymbol{0}})$ est sans pente et donc par ce qui précède que $(\boldsymbol{H},\mathcal{M}_{\boldsymbol{\alpha},{\boldsymbol{k}}})$ l'est.

Pour démontrer la deuxième partie de la proposition on commence par noter
\[U_{\boldsymbol{\beta}}(\mathcal{M}_{\boldsymbol{\alpha},{\boldsymbol{k}}}):=\bigoplus_{\boldsymbol{0}\leq\boldsymbol{\ell}\leq{\boldsymbol{k}}}V_{\boldsymbol{\alpha}+\boldsymbol{\beta}+\boldsymbol{1}}
 \left(\mathcal{M}[\frac{1}{t_1...t_p}]\right)e_{\boldsymbol{\alpha},\boldsymbol{\ell}}
\]
et on va montrer que c'est une bonne $V$-multifiltration qui satisfait à toutes les propriétés caractéristiques de la multifiltration de Malgrange-Kashiwara. Soit $m\in \mathcal{M}$, $\boldsymbol{\ell}\in \mathbb{N}^p$, $\beta\in \mathbb{C}$ et $1\leq i\leq p$. On a localement
\begin{equation}\label{calcul}
(t_i\partial_i+\beta)(m\otimes e_{\boldsymbol{\alpha},\boldsymbol{\ell}})=((t_i\partial_i+\beta+\alpha_i+1)m)\otimes e_{\boldsymbol{\alpha},\boldsymbol{\ell}}+m\otimes e_{\boldsymbol{\alpha},\boldsymbol{\ell-1}_i}
\end{equation}
et pour tout $\boldsymbol{n}\in\mathbb{Z}^p$
\[\boldsymbol{t}^{\boldsymbol{n}}(m\otimes e_{\boldsymbol{\alpha},\boldsymbol{\ell}})=(\boldsymbol{t}^{\boldsymbol{n}}m)\otimes e_{\boldsymbol{\alpha},\boldsymbol{\ell}}.\]
Ceci permet de montrer que $U_\bullet(\mathcal{M}_{\boldsymbol{\alpha},{\boldsymbol{k}}})$ est une $V$-multifiltration de $\mathcal{M}_{\boldsymbol{\alpha},{\boldsymbol{k}}}$ (c'est-à-dire que cette multifiltration vérifie $V_{\boldsymbol{\ell}}\mathcal{D}_X.U_{\boldsymbol{\beta}}(\mathcal{M}_{\boldsymbol{\alpha},{\boldsymbol{k}}})\subset U_{\boldsymbol{\beta}+\boldsymbol{\ell}}(\mathcal{M}_{\boldsymbol{\alpha},{\boldsymbol{k}}})$ pour tout $\boldsymbol{\beta}\in\mathbb{C}^p$ et pour tout $\boldsymbol{\ell}\in\mathbb{Z}^p$). 

Pour montrer que c'est une bonne $V$-multifiltration on fixe $\boldsymbol{\beta}\in\mathbb{C}^p$ et on montre que la $V$-multifiltration indexée par $\mathbb{Z}^p$, $U_{\boldsymbol{\beta}+\bullet}(\mathcal{M}_{\boldsymbol{\alpha},{\boldsymbol{k}}})$, est une bonne $V$-multifiltration de $\mathcal{M}_{\boldsymbol{\alpha},{\boldsymbol{k}}}$. Comme  la $V$-multifiltration indexée par $\mathbb{Z}^p$, $V_{\boldsymbol{\alpha}+\boldsymbol{\beta}+\bullet+\boldsymbol{1}}\left(\mathcal{M}[\frac{1}{t_1...t_p}]\right)$, est une bonne $V$-multifiltration elle est engendrée localement par un nombre fini de sections $\{m_j\}_ {j\in J}$. Si ${\boldsymbol{k}}=\boldsymbol{0}$ l'égalité \eqref{calcul} permet de montrer que les sections $\{m_j\otimes e_{\boldsymbol{\alpha},\boldsymbol{0}}\}_{j\in J}$ engendrent la $V$-multifiltration $U_{\boldsymbol{\beta}+\bullet}(\mathcal{M}_{\boldsymbol{\alpha},\boldsymbol{0}})$. On peut alors montrer par récurrence, en considérant la suite exacte \eqref{exactm} et l'égalité \eqref{calcul}, que pour tout ${\boldsymbol{k}}\in\mathbb{N}^p$ les sections $m_j\otimes e_{\boldsymbol{\alpha},\boldsymbol{\ell}}$, pour
$j\in J$ et $0\leq\boldsymbol{\ell}\leq{\boldsymbol{k}}$, engendrent la $V$-multifiltration $U_{\boldsymbol{\beta}+\bullet}(\mathcal{M}_{\boldsymbol{\alpha},{\boldsymbol{k}}})$. C'est donc une bonne $V$-multifiltration de $\mathcal{M}_{\boldsymbol{\alpha},{\boldsymbol{k}}}$.

On fixe maintenant $\boldsymbol{\beta}\in \mathbb{C}^p$ et on va construire, pour tout $1\leq i\leq p$, un polynôme $b_i(s)$ qui satisfait à 
\[b_i(t_i\partial_i)U_{\boldsymbol{\beta}}(\mathcal{M}_{\boldsymbol{\alpha},{\boldsymbol{k}}}) \subset U_{\boldsymbol{\beta-1}_i}(\mathcal{M}_{\boldsymbol{\alpha},{\boldsymbol{k}}}).\]
Par définition de la multifiltration de Malgrange-Kashiwara on peut choisir, pour tout $1\leq i\leq p$, un polynôme $c_i(s)$ vérifiant 
\[c_i(t_i\partial_i+\alpha_i+\beta_i+1)V_{\boldsymbol{\alpha}+\boldsymbol{\beta}+\boldsymbol{1}}\left(\mathcal{M}[\frac{1}{t_1...t_p}]\right)\subset  V_{\boldsymbol{\alpha}+\boldsymbol{\beta}+\boldsymbol{1-1}_i}\left(\mathcal{M}[\frac{1}{t_1...t_p}]\right)\]
et ayant ses racines dans l'intervalle $[-1,0[$. Soit $m\in V_{\boldsymbol{\alpha}+\boldsymbol{\beta}+\boldsymbol{1}}\left(\mathcal{M}[\frac{1}{t_1...t_p}]\right)$, l'égalité \eqref{calcul} permet de montrer que 
\[c_i(t_i\partial_i+\beta_i)(m\otimes e_{\boldsymbol{\alpha},\boldsymbol{\ell}})=(c_i(t_i\partial_i+\beta_i+\alpha_i+1)m)\otimes e_{\boldsymbol{\alpha},\boldsymbol{\ell}}+\tilde{m}\]
où $\tilde{m}\in  U_{\boldsymbol{\beta}}(\mathcal{M}_{\boldsymbol{\alpha},\boldsymbol{k-1}_i})$ si on pose $U_{\boldsymbol{\beta}}(\mathcal{M}_{\boldsymbol{\alpha},\boldsymbol{\ell}})=0$ pour $l_i<0$. On peut donc construire par récurrence un polynôme $b_{i,m}(s)$ ayant ses racines dans l'intervalle $[-1,0[$ et vérifiant
\[b_{i,m}(t_i\partial_i+\beta_i)(m\otimes e_{\boldsymbol{\alpha},\boldsymbol{\ell}})\in U_{\boldsymbol{\beta-1}_i}(\mathcal{M}_{\boldsymbol{\alpha},{\boldsymbol{k}}}).\]
Comme $U_{\boldsymbol{\beta}}(\mathcal{M}_{\boldsymbol{\alpha},{\boldsymbol{k}}})$ est localement engendré par un nombre fini de sections de la forme $m\otimes e_{\boldsymbol{\alpha},\boldsymbol{\ell}}$ pour $0\leq\boldsymbol{\ell}\leq{\boldsymbol{k}}$ on peut construire $b_i(s)$ ayant ses racines dans $[-1,0[$ tel que 
\[b_i(t_i\partial_i+\beta_i)U_{\boldsymbol{\beta}}(\mathcal{M}_{\boldsymbol{\alpha},{\boldsymbol{k}}}) \subset U_{\boldsymbol{\beta-1}_i}(\mathcal{M}_{\boldsymbol{\alpha},{\boldsymbol{k}}}).\]
Les racines du polynôme de Bernstein-Sato de la $V$-multifiltration $U_\bullet(\mathcal{M}_{\boldsymbol{\alpha},\boldsymbol{k-1}_i})$ sont donc dans l'intervalle $[-1,0[$, ce qui permet de conclure que c'est bien la $V$-multifiltration de Malgrange-Kashiwara:
 \[V_{\boldsymbol{\beta}}(\mathcal{M}_{\boldsymbol{\alpha},{\boldsymbol{k}}})=\bigoplus_{\boldsymbol{0}\leq\boldsymbol{\ell}\leq{\boldsymbol{k}}}V_{\boldsymbol{\alpha}+\boldsymbol{\beta}+\boldsymbol{1}}
 \left(\mathcal{M}[\frac{1}{t_1...t_p}]\right)e_{\boldsymbol{\alpha},\boldsymbol{\ell}}.
\]

\end{proof}

\section{Morphisme de comparaison}\label{morphcomp}

On va construire un morphisme de comparaison entre les cycles proches algébriques de $\mathcal{M}$ et les cycles proches topologiques de $\mathbf{DR}(\mathcal{M})$ relativement à l'application 
 \[\begin{array}{cccc}
     \pi:& X & \to & \mathbb{C}^p\\
       & (\boldsymbol{x},t_1,...,t_p) & \mapsto & (t_1,...,t_p).
     \end{array}\] 
On établira le lien avec la composition du morphisme de comparaison relatif aux $r$ premières coordonnées $t_i$ et de celui relatif aux $p-r$ coordonnées $t_i$ suivantes pour $1<r<p$.

\subsection{Comparaison avec les gradués}

Commençons par donner deux définitions.

\begin{de}\label{dagger}

Soit $\mathcal{M}$ un $\mathcal{D}_X$-module tel que le couple $(\boldsymbol{H},\mathcal{M})$ soit sans pente. On considère la famille $\{\textup{gr}_{\boldsymbol{k}}(\mathcal{M}),\partial_i\}_{{\boldsymbol{k}}\in\{0,1\}^p,1\leq i\leq p}$ composée des objets $\textup{gr}_{\boldsymbol{k}}(\mathcal{M})$ pour ${\boldsymbol{k}}\in\{0,1\}^p$ et des morphismes $\partial_i:\textup{gr}_{\boldsymbol{k}}(\mathcal{M})\to \textup{gr}_{\boldsymbol{k+1}_i}(\mathcal{M})$. On définit
\[i^\dagger\mathcal{M}:=\restriction{s(\textup{Cube}(\textup{gr}_\bullet(\mathcal{M})))}{X_0}\]
où $s(.)$ et $\textup{Cube}(.)$ sont les foncteurs définis dans l'appendice \ref{simple} et \ref{cube} et $X_0=\pi^{-1}(0)$.

\end{de}

Par exemple pour $p=2$ on a 
\[
\begin{array}{cccccccclll}
i^\dagger\mathcal{M}= & 0 \to & \restriction{\textup{gr}_{-1,-1}(\mathcal{M})}{X_0} & \to & \restriction{\textup{gr}_{0,-1}(\mathcal{M})}{X_0}  \bigoplus  \restriction{\textup{gr}_{-1,0}(\mathcal{M})}{X_0} & \to & \restriction{\textup{gr}_{0,0}(\mathcal{M})}{X_0} & \to 0 \\\\
& & m & \mapsto &  (\partial_1m  , - \partial_2m)  \\\\
&&&& (m_1  ,  m_2) & \mapsto & \partial_2m_1+\partial_1m_2 & .
\end{array}
\]

\begin{de}\label{sharp}

De la même manière que pour la définition précédente on considère la famille $\{{V}_{\boldsymbol{k}}(\mathcal{M}),\partial_i\}_{{\boldsymbol{k}}\in\{0,1\}^p,1\leq i\leq p}$ composée des objets ${V}_{\boldsymbol{k}}(\mathcal{M})$ pour ${\boldsymbol{k}}\in\{0,1\}^p$ et des morphismes $\partial_i:{V}_{\boldsymbol{k}}(\mathcal{M})\to {V}_{\boldsymbol{k+1}_i}(\mathcal{M})$. On définit
\[i^\#\mathcal{M}:=\restriction{s(\textup{Cube}({V}_\bullet(\mathcal{M})))}{X_0}\]
où $X_0=\pi^{-1}(0)$.

\end{de}

\begin{Rem}

\begin{enumerate}

\item Notons que si on considère la famille $\mathscr{M}:=\{\mathcal{M},\partial_i\}_{{\boldsymbol{k}}\in\{0,1\}^p,1\leq i\leq p}$ on a 
\[\restriction{s(\textup{Cube}(\mathscr{M}))}{X_0}\simeq \restriction{\mathbf{DR}_{X/X_0}(\mathcal{M})}{X_0}
\]
où l'on considère la projection 
\[\begin{array}{cccc}
     \tau :& X & \to & X_0\\
       & (\boldsymbol{x},t_1,...,t_p) & \mapsto & (\boldsymbol{x},0,...,0).
     \end{array}\] 
\item On étend ces définitions aux complexes en commençant par appliquer $\textup{Cube}(.)$ en chaque degré puis en prenant le complexe simple associé à l'hypercomplexe obtenu. On note encore $i^\#$ et $i^\dagger$ ces foncteurs appliqués aux complexes.

\end{enumerate}

\end{Rem}
D'après la remarque précédente les morphismes naturels pour tout $\boldsymbol{k}\in\{0,1\}^p$ 
\[\textup{gr}_{\boldsymbol{k}}(\mathcal{M}) \leftarrow
{V}_{\boldsymbol{k}}(\mathcal{M}) \to 
\mathcal{M}\]
induisent les morphismes de complexes 
\begin{equation}\label{compgrad}
\boxed{i^\dagger\mathcal{M} \leftarrow
i^\#\mathcal{M} \to
\mathbf{DR}_{X/X_0}(\mathcal{M})}
\end{equation}
où l'on omet de noter la restriction de $\mathbf{DR}_{X/X_0}(\mathcal{M})$ à $X_0$. Soit $I=\{1,...,r\}\subset\{1,...,p\}$ les $r$ premiers entiers pour $r<p$, on note 
\[\begin{array}{cccc}
     \pi_I:& X & \to & \mathbb{C}^r\\
       & (\boldsymbol{x},t_1,...,t_p) & \mapsto & (t_1,...,t_r)
     \end{array}\] 
et $X_0^I:=\pi_I^{-1}(0)$. On note $V_\bullet^{I}$ la $V$-multifiltration par rapport aux fonctions $t_1,...,t_r$. La $V$-multifiltration de Malgrange-Kashiwara de $\mathcal{M}$ induit une $V^{I^c}$-multifiltration du $\mathcal{D}_{X_0^I}$-module  $\text{gr}_{\boldsymbol{\alpha}_I}^{I}(\mathcal{M})$ pour tout $\boldsymbol{\alpha}_I\in\mathbb{C}^r$. Pour tout $\boldsymbol{\alpha}\in\mathbb{C}^p$ on a le diagramme commutatif suivant 
\begin{equation}\xymatrix{
\text{gr}_{\boldsymbol{\alpha}}\mathcal{M} \ar[d] & &
 V_{\boldsymbol{\alpha}}\mathcal{M} \ar@{=}[d] \ar[ll] \\
 \text{gr}_{\boldsymbol{\alpha}_{I^c}}^{I^c}
\left(\text{gr}_{\boldsymbol{\alpha}_I}^I\mathcal{M}\right) &
V_{\boldsymbol{\alpha}_{I^c}}^{I^c}
\left(\text{gr}_{\boldsymbol{\alpha}_I}^I\mathcal{M}\right) \ar[l] &
V_{\boldsymbol{\alpha}_{I^c}}^{I^c}
\left(V_{\boldsymbol{\alpha}_I}^I\mathcal{M}\right). \ar[l]}
\label{coupegradV}\end{equation} 

On définit les foncteurs $i_I^\dagger$ et $i_I^\#$ en considérant respectivement les familles $\{\text{gr}_{{\boldsymbol{k}}_I}(\mathcal{M}'),\partial_i\}_{{\boldsymbol{k}}_I\in\{0,1\}^r,1\leq i\leq r}$ et $\{{V}_{{\boldsymbol{k}}_I}(\mathcal{M}'),\partial_i\}_{{\boldsymbol{k}}_I\in\{0,1\}^r,1\leq i\leq r}$. On définit de manière analogue les foncteurs $i_{I^c}^\dagger$ et $i_{I^c}^\#$ appliqués à la catégorie des $\mathcal{D}_{X_0^I}$-modules en considérant la projection \[\begin{array}{cccc}
     \restriction{\pi_{I^c}}{X_0^I}:& X_0^I & \to & \mathbb{C}^{p-r}\\
       & (\boldsymbol{x},t_{p-r},...,t_p) & \mapsto & (t_{p-r},...,t_p).
     \end{array}\] 
Les propriétés des hypercomplexes, du foncteur $s(.)$ et le diagramme commutatif \eqref{coupegradV} pour $\boldsymbol{\alpha}\in\{0,1\}^p$ fournissent le diagramme commutatif suivant

\begin{equation}\label{commutgradV}\xymatrix{i^\dagger\mathcal{M}  \ar[d]
& & 
i^\#\mathcal{M} \ar[ll] \ar@{=}[d] \ar[rr]
& &
\mathbf{DR}_{X/X_0}\mathcal{M} \ar@{=}[d]
\\
i_{I_c}^\dagger(i_I^\dagger\mathcal{M})
&
i_{I_c}^\#(i_I^\dagger\mathcal{M}) \ar[l]
&
i_{I_c}^\#(i_I^\#\mathcal{M}) \ar[l] \ar[r]
&
i_{I^c}^\#(
\mathbf{DR}_{X/X_0^I}\mathcal{M}) \ar[r]
&
\mathbf{DR}_{X_0^I/X_0}(
\mathbf{DR}_{X/X_0^I}\mathcal{M}).
}\end{equation}

\subsection{Le morphisme <<Nils>>}

D'après la proposition \ref{VNilsson} on a 
\[\textup{gr}_{\boldsymbol{-1}}(\mathcal{M}_{\boldsymbol{\alpha},{\boldsymbol{k}}})=\bigoplus_{\boldsymbol{0}\leq\boldsymbol{\ell}\leq{\boldsymbol{k}}}
\textup{gr}_{\boldsymbol{\alpha}}
 \left(\mathcal{M}[\frac{1}{t_1...t_p}]\right)e_{\boldsymbol{\alpha},\boldsymbol{\ell}}.
\]
La proposition \ref{loca} assure que pour $\boldsymbol{\alpha}\in[-1,0[^p$ on a l'isomorphisme

\[\text{gr}_{\boldsymbol{\alpha}}\left(\mathcal{M}\right) \simeq \textup{gr}_{\boldsymbol{\alpha}}
 \left(\mathcal{M}[\frac{1}{t_1...t_p}]\right).\]
On définit alors le morphisme suivant 
\[\begin{array}{cccc}
\Phi: & \text{gr}_{\boldsymbol{\alpha}}\left(\mathcal{M}\right) & \longrightarrow & \text{gr}_{-\boldsymbol{1}}(\mathcal{M}_{\boldsymbol{\alpha},{\boldsymbol{k}}}) \\
& m & \longmapsto & \displaystyle\sum_{0\leq \boldsymbol{\ell}\leq{\boldsymbol{k}}}\left[(-1)^{\ell_1+...+\ell_p}(t_1\partial_1+\alpha_1+1)^{\ell_1}...(t_p\partial_p+\alpha_p+1)^{\ell_p}m\right]\otimes e_{\boldsymbol{\alpha},\boldsymbol{\ell}}
\end{array}\]
qui induit un morphisme de complexes 
\[\boxed{\mathbf{Nils}:\text{gr}_{\boldsymbol{\alpha}}\left(\mathcal{M}\right) \to i^\dagger\mathcal{M}_{\boldsymbol{\alpha}}}\]
où l'on identifie $\text{gr}_{\boldsymbol{\alpha}}\left(\mathcal{M}\right)$ avec un complexe concentré en degré zéro et où $\mathcal{M}_{\boldsymbol{\alpha}}$ est la limite inductive des $\mathcal{M}_{\boldsymbol{\alpha},\boldsymbol{k}}$ prise sur $\boldsymbol{k}\in\mathbb{N}^p$. 

\begin{Rem}

Remarquons ici que $\mathcal{M}_{\boldsymbol{\alpha}}$ n'est pas un $\mathcal{D}_X$-module de type fini. Mais le fait qu'il soit limite des $\mathcal{M}_{\boldsymbol{\alpha},\boldsymbol{k}}$ et que les couples $(\boldsymbol{H},\mathcal{M}_{\boldsymbol{\alpha},{\boldsymbol{k}}})$ soient sans pente suffit pour le reste de la construction et pour le théorème de comparaison.

\end{Rem}

En utilisant la définition \ref{Nilsson} on obtient
\[ \mathcal{O}_X \otimes_{\pi^{-1}\mathcal{O}_{\mathbb{C}^p}} \pi^{-1}(\mathcal{N}_{\boldsymbol{\alpha,k}}) \simeq 
\left(\mathcal{O}_X 
\otimes_{\pi^{-1}_I\mathcal{O}_{\mathbb{C}^r}}
\pi^{-1}_I(\mathcal{N}_{\boldsymbol{\alpha_I,k_I}})\right) 
\otimes_{\mathcal{O}_X}
\left(\mathcal{O}_X 
\otimes_{\pi^{-1}_{I^c}\mathcal{O}_{\mathbb{C}^{p-r}}}
\pi^{-1}_{I^c}(\mathcal{N}_{\boldsymbol{\alpha_{I^c},k_{I^c}}})\right).\]
On déduit de cet isomorphisme et de la définition du morphisme $\Phi$ le diagramme commutatif suivant

\begin{equation}\label{commutNils}
\xymatrix{\text{gr}_{\boldsymbol{\alpha}}\mathcal{M} \ar[rr]^{\mathbf{Nils}} \ar[dd]
&& i^\dagger\mathcal{M}_{\boldsymbol{\alpha}} \ar[d] \\
&& i_{I_c}^\dagger(i_I^\dagger\mathcal{M}_{\boldsymbol{\alpha}}) \ar@{=}[d] \\
\text{gr}_{\boldsymbol{\alpha}_{I^c}}^{I^c}
\left(\text{gr}_{\boldsymbol{\alpha}_I}^I\mathcal{M}\right) \ar[r]
& \text{gr}_{\boldsymbol{\alpha}_{I^c}}^{I^c}
\left(i_I^\dagger\mathcal{M}_{\boldsymbol{\alpha}_I}\right) \ar[r]
& i_{I^c}^\dagger\left[(i_I^\dagger\mathcal{M}_{\boldsymbol{\alpha}_I})
_{\boldsymbol{\alpha}_{I^c}}\right].
}\end{equation}

\subsection{Le morphisme <<Topo>>}

Rappelons le diagramme commutatif utilisé pour définir les cycles proches topologiques:
\[\xymatrix{ \pi^{-1}(0) \ar[r]^i \ar[d] & X \ar[d]^\pi & X^* \ar[l]_j  \ar[d]^{\restriction{\pi}{X^*}} & \widetilde{X} \ar[l]_p \ar[d]^{\tilde{\pi}}\\
\{0\} \ar[r]^i & \mathbb{C}^p & (\mathbb{C}^*)^p \ar[l]_j & \widetilde{(\mathbb{C}^*)^p}. \ar[l]_p}\]

\begin{lemma}\label{cyclepro}

Soit $\boldsymbol{\alpha}\in \mathbb{C}^p$, il existe un morphisme naturel 
\[\boxed{\mathbf{Topo}:\mathbf{DR}_X(\mathcal{M }_{\boldsymbol{\alpha}})\to \Psi_{\pi}\mathbf{DR}_X(\mathcal{M}).}\]

\end{lemma}

\begin{proof}

Par définition, $\mathcal{M}_{\boldsymbol{\alpha}}=\mathcal{M}\otimes_{\pi^{-1}\mathcal{O}_{\mathbb{C}^p}} \pi^{-1}\mathcal{N}_{\boldsymbol{\alpha}}$, or on a une inclusion $\mathcal{N}_{\boldsymbol{\alpha}}\subset j_*p_*p^{-1}\mathcal{O}_{(\mathbb{C}^*)^p}$ dans le faisceau des fonctions holomorphes multiformes. Par fonctorialité on a donc le morphisme:
\[\mathbf{DR}_X(\mathcal{M }_{\boldsymbol{\alpha}})\to \mathbf{DR}_X(\mathcal{M}\otimes \pi^{-1}j_*p_*p^{-1}\mathcal{O}_{(\mathbb{C}^*)^p}).\]

L'adjonction des foncteurs  image inverse et image directe fournit un morphisme de foncteurs $\pi^{-1}(j\circ p)_*\to (j\circ p)_*\tilde{\pi}^{-1}$. Ceci donne le morphisme:

\[\begin{array}{cccc}\mathbf{DR}_X(\mathcal{M}\otimes \pi^{-1}j_*p_*p^{-1}\mathcal{O}) & \to & \mathbf{DR}_X(\mathcal{M}\otimes j_*p_*\tilde{\pi}^{-1}p^{-1}\mathcal{O}) \\
 &  & = \mathbf{DR}_X(\mathcal{M}\otimes j_*p_*p^{-1}\restriction{\pi}{X^*}^{-1}\mathcal{O}).
\end{array}\]
Par adjonction on a le morphisme:
\[\begin{array}{ccl}
\mathbf{DR}_X(\mathcal{M}\otimes j_*p_*p^{-1}\restriction{\pi}{X^*}^{-1}\mathcal{O}) & \to & \boldsymbol{R}j_*j^{-1}\mathbf{DR}_X(\mathcal{M}
\otimes j_*p_*p^{-1}\restriction{\pi}{X^*}^{-1}\mathcal{O})  \\
& = & \boldsymbol{R}j_*\mathbf{DR}_X(j^{-1}\mathcal{M}
\otimes j^{-1}j_*p_*p^{-1}\pi^{-1}\mathcal{O}) \\
& = & \boldsymbol{R}j_*\mathbf{DR}_X(j^{-1}\mathcal{M}
\otimes p_*p^{-1}\pi^{-1}\mathcal{O}).
\end{array}\]
On applique ensuite le morphisme (2.3.21) de \cite{kash} (formule de projection) à la fonction $p$, en considérant le fait que $p_*$ est un foncteur exact car $p$ est à fibres discrètes. Par fonctorialité on a alors le morphisme suivant:
\[\begin{array}{ccl}
\boldsymbol{R}j_*\mathbf{DR}_X(j^{-1}\mathcal{M}
\otimes p_*p^{-1}\pi^{-1}\mathcal{O}) & \to & 
\boldsymbol{R}j_*\mathbf{DR}_X(p_*p^{-1}(j^{-1}\mathcal{M}
\otimes \pi^{-1}\mathcal{O}))\\
& = & \boldsymbol{R}j_*\mathbf{DR}_X(p_*p^{-1}j^{-1}\mathcal{M}).
\end{array}\]

Sachant que $\mathbf{DR}_X\mathcal{M}=\Omega^n\overset{\mathbb{L}}{\otimes}_{\mathcal{D}_X} \mathcal{M}$, on peut appliquer le morphisme (2.6.21) de \cite{kash} à $p$ (formule de projection) et on obtient le morphisme:
\[\begin{array}{llc}
\boldsymbol{R}j_*\mathbf{DR}_X(p_*p^{-1}j^{-1}\mathcal{M}) & \to & \boldsymbol{R}j_*p_*\mathbf{DR}_X(p^{-1}j^{-1}\mathcal{M})\\
& = & \boldsymbol{R}j_*p_*p^{-1}j^{-1}\mathbf{DR}_X(\mathcal{M}).
\end{array}\]

Si l'on compose tous les morphismes naturels que l'on vient de construire on obtient bien le morphisme naturel attendu:
\[\mathbf{DR}_X(\mathcal{M }_{\boldsymbol{\alpha}})\to \Psi_{\pi}\mathbf{DR}_X(\mathcal{M}).\]

\end{proof}

La naturalité de ce morphisme ainsi que la définition du morphisme \eqref{coupeproche}
\[
\Psi_\pi\mathbf{DR}_X(\mathcal{M})\to \Psi_{\pi_{I^c}}(\Psi_{\pi_I}\mathbf{DR}_X(\mathcal{M}))
\]
permettent de montrer que le diagramme suivant est commutatif

\begin{equation}\label{commutTopo}\xymatrix{\mathbf{DR}_{X}\mathcal{M}_{\boldsymbol{\alpha}} \ar@{=}[dd] \ar[rr]^{\mathbf{Topo}} && 
\Psi_{\pi}\mathbf{DR}_{X}\mathcal{M} \ar[dd]
\\\\
\mathbf{DR}_{X}\left[(\mathcal{M}_{\boldsymbol{\alpha}_{I}})_{\boldsymbol{\alpha}_{I^c}}\right] \ar[r] &
\Psi_{\pi_{I^c}}(
\mathbf{DR}_{X}\mathcal{M}_{\boldsymbol{\alpha}_{I}}) \ar[r] & \Psi_{\pi_{I^c}}(\Psi_{\pi_{I}}\mathbf{DR}_{X}\mathcal{M}).
}\end{equation}

\subsection{Le morphisme de comparaison}

En combinant les morphismes \eqref{compgrad}, $\mathbf{Nils}$ et $\mathbf{Topo}$ on obtient la suite de morphismes suivante
\begin{multline}\label{comparaison} \mathbf{DR}_{X_0}
\Psi_{\mathbf{H}}\left(\mathcal{M}\right) 
\xrightarrow{\mathbf{Nils}}
\bigoplus_{\boldsymbol{\alpha}\in[-1,0[^p}\mathbf{DR}_{X_0}
i^\dagger\mathcal{M}_{\boldsymbol{\alpha}} 
\leftarrow 
\bigoplus_{\boldsymbol{\alpha}\in[-1,0[^p}\mathbf{DR}_{X_0}
i^\#\mathcal{M}_{\boldsymbol{\alpha}} 
\to \\
\bigoplus_{\boldsymbol{\alpha}\in[-1,0[^p}
\mathbf{DR}_X(\mathcal{M}_{\boldsymbol{\alpha}})
\xrightarrow{\mathbf{Topo}}
\Psi_{\pi}\mathbf{DR}_X(\mathcal{M}).
\end{multline}
On a appliqué les morphisme \eqref{compgrad} à $\mathcal{M}_{\boldsymbol{\alpha}}$, on a ensuite appliqué le foncteur $\mathbf{DR}_{X_0}$ et on a pris la somme sur $\boldsymbol{\alpha}\in ]-1,0]^p$ en utilisant la définition
\[\Psi_{\mathbf{H}}\left(\mathcal{M}\right) := \bigoplus_{\boldsymbol{\alpha}\in[-1,0[^p}\textup{gr}_{\boldsymbol{\alpha}}(\mathcal{M}).\]

\begin{theorem}\label{thcomparaison}

Si le couple $(\mathbf{H},\mathcal{M})$ est sans pente alors les morphismes \eqref{comparaison} sont des isomorphismes qui commutent aux endomorphismes de monodromie $T_i$, on obtient l'isomorphisme de comparaison
\[\mathbf{DR}_{X_0}
\Psi_{\mathbf{H}}\left(\mathcal{M}\right)\simeq
\Psi_{\pi}\mathbf{DR}_X(\mathcal{M}).\]
De plus si $I=\{1,...,r\}\subset\{1,...,p\}$ et si l'on applique successivement cet isomorphisme de comparaison par rapport aux familles d'hypersurfaces $\mathbf{H}_I$ et $\mathbf{H}_{I^c}$ le résultat ne dépend pas de l'ordre dans lequel on applique l'isomorphisme. Autrement dit le diagramme suivant est commutatif
\[
\xymatrix{\mathbf{DR}_{X_0}\Psi_{\mathbf{H}_{I^c}}(\Psi_{\mathbf{H}_{I}}\mathcal{M})  \ar@{-}[d]^{\simeq} &
\mathbf{DR}_{X_0}
\Psi_{\mathbf{H}}\left(\mathcal{M}\right) \ar[l]_-{\sim} \ar[r]^-{\sim} \ar@{-}[d]^{\simeq} &
\mathbf{DR}_{X_0}\Psi_{\mathbf{H}_I}(\Psi_{\mathbf{H}_{I^c}}\mathcal{M})  \ar@{-}[d]^{\simeq} \\
\Psi_{\pi_{I^c}}(\Psi_{\pi_I}\mathbf{DR}_X(\mathcal{M})) &
\Psi_{\pi}\mathbf{DR}_X(\mathcal{M}) \ar[l]_-{\sim} \ar[r]^-{\sim} & 
\Psi_{\pi_{I}}(\Psi_{\pi_{I^c}}\mathbf{DR}_X(\mathcal{M})).
}\]

\end{theorem}

\begin{proof}

On raisonne par récurrence sur le nombre $p$ d'hypersurfaces dans $\mathbf{H}$, le cas $p=1$ est traité par Ph. Maisonobe et Z. Mebkhout dans \cite[théorème 5.3-2]{MM}  ou par Morihiko Saito dans \cite[lemmes 3.4.4 et 3.4.5]{HM1}. 

Pour $p>1$, soit $I=\{1,...,r\}\subset\{1,...,p\}$ avec $1<r<p$, on va considérer les diagrammes commutatifs \eqref{commutgradV}, \eqref{commutNils} et \eqref{commutTopo}. L'hypothèse sans pente permet d'appliquer la proposition \ref{maisonalg} (resp. \ref{maisontopo}) qui assure que les flèches verticales des diagrammes \eqref{commutgradV} et \eqref{commutNils} (resp. \eqref{commutTopo}) sont des isomorphismes. La commutativité de ces diagrammes permet de se ramener aux cas de $r$ et $p-r$ hypersurfaces en appliquant successivement les deux isomorphismes de comparaison obtenus par récurrence. La commutativité donne alors également directement la deuxième partie du théorème.  

\end{proof}

Pour un morphisme $\boldsymbol{f}:X\to \mathbb{C}^p$, l'inclusion du graphe de $\boldsymbol{f}$ permet de donner une version générale de ce théorème:

\begin{corollary}\label{thcomparaison2}

Soit $\boldsymbol{f}:X\to \mathbb{C}^p$ un morphisme d'espaces analytiques complexes réduits et $\mathcal{M}$ un $\mathcal{D}_X$-module holonome régulier tel que le couple $(\boldsymbol{H},{i_{\boldsymbol{f}}}_+\mathcal{M})$ soit sans pente. On a un isomorphisme de comparaison 
\[\mathbf{DR}_{X}
\Psi^{\textup{alg}}_{\boldsymbol{f}}\left(\mathcal{M}\right)\simeq
\Psi_{\boldsymbol{f}}\mathbf{DR}_X(\mathcal{M}).\]
De plus si $I=\{1,...,r\}\subset\{1,...,p\}$ et si l'on applique successivement cet isomorphisme de comparaison par rapport aux fonctions $\boldsymbol{f}_I$ et $\boldsymbol{f}_{I^c}$ le résultat ne dépend pas de l'ordre dans lequel on applique l'isomorphisme. Autrement dit le diagramme suivant est commutatif
\[
\xymatrix{\mathbf{DR}_{X}\Psi^{\textup{alg}}_{\boldsymbol{f}_{I^c}}(\Psi^{\textup{alg}}_{\boldsymbol{f}_{I}}\mathcal{M})  \ar@{-}[d]^{\simeq} &
\mathbf{DR}_{X}
\Psi^{\textup{alg}}_{\boldsymbol{f}}\left(\mathcal{M}\right) \ar[l]_-{\sim} \ar[r]^-{\sim} \ar@{-}[d]^{\simeq} &
\mathbf{DR}_{X}\Psi^{\textup{alg}}_{\boldsymbol{f}_I}(\Psi^{\textup{alg}}_{\boldsymbol{f}_{I^c}}\mathcal{M})  \ar@{-}[d]^{\simeq} \\
\Psi_{\boldsymbol{f}_{I^c}}(\Psi_{\boldsymbol{f}_I}\mathbf{DR}_X(\mathcal{M})) &
\Psi_{\boldsymbol{f}}\mathbf{DR}_X(\mathcal{M}) \ar[l]_-{\sim} \ar[r]^-{\sim} & 
\Psi_{\boldsymbol{f}_{I}}(\Psi_{\boldsymbol{f}_{I^c}}\mathbf{DR}_X(\mathcal{M})).
}\]

\end{corollary}

\begin{proof}

On applique le théorème \ref{thcomparaison} à ${i_{\boldsymbol{f}}}_+\mathcal{M}$, on obtient l'isomorphisme 
\[\mathbf{DR}_{X_0}
\Psi_{\mathbf{H}}\left({i_{\boldsymbol{f}}}_+\mathcal{M}\right)\simeq
\Psi_{\pi}\mathbf{DR}_{X\times\mathbb{C}^p}({i_{\boldsymbol{f}}}_+\mathcal{M}).\]
où $\pi:X\times \mathbb{C}^p \to \mathbb{C}^p$ est la projection. On applique le foncteur ${i_{\boldsymbol{f}}}^{-1}$ à cette isomorphisme. On observe qu'un théorème de changement de base propre donne l'isomorphisme de foncteur $\Psi_{\boldsymbol{f}}{i_{\boldsymbol{f}}}^{-1}\simeq {i_{\boldsymbol{f}}}^{-1}\Psi_{{\pi}}$. On en déduit l'isomorphisme
\[{i_{\boldsymbol{f}}}^{-1}\mathbf{DR}_{X_0}
\Psi_{\mathbf{H}}\left({i_{\boldsymbol{f}}}_+\mathcal{M}\right)\simeq
\Psi_{\boldsymbol{f}}{i_{\boldsymbol{f}}}^{-1}\mathbf{DR}_{X\times\mathbb{C}^p}({i_{\boldsymbol{f}}}_+\mathcal{M}).\]
On déduit enfin de l'équivalence de Kashiwara appliquée à l'injection du graphe de $\boldsymbol{f}$ dans $X\times \mathbb{C}^p$ l'isomorphisme attendu
\[\mathbf{DR}_{X}
\Psi^{\textup{alg}}_{\boldsymbol{f}}\left(\mathcal{M}\right)\simeq
\Psi_{\boldsymbol{f}}\mathbf{DR}_X(\mathcal{M}).\]
La suite du corollaire se démontre de la même manière.

\end{proof}

On déduit en particulier de ce corollaire que, dans le cas sans pente, si l'on applique l'isomorphisme de comparaison par rapport aux fonctions $f_1,...,f_p$ l'une après l'autre l'isomorphisme
\[\mathbf{DR}_{X}\left(\Psi^{\textup{alg}}_{f_{\sigma(p)}}\left(...\Psi^{\textup{alg}}_{f_{\sigma(2)}}\left(\Psi^{\textup{alg}}_{f_{\sigma(1)}}\mathcal{M}\right)
\right)\right)
\simeq 
\Psi_{f_{\sigma(p)}}\left(...\Psi_{f_{\sigma(2)}}\left(\Psi_{f_{\sigma(1)}}\mathbf{DR}_{X}(\mathcal{M})\right)\right)\]
ne dépend pas de la permutation $\sigma$ de $\{1,...,p\}$.

\appendix

\section{Hypercomplexes}

On définit ici les $n$-hypercomplexes qui correspondent aux complexes $n^{uple}$ naïfs introduits par P. Deligne au paragraphe 0.4 de \cite{DeligneSGA4}.

\begin{de}
 
Soit $\mathcal{C}$ une catégorie abélienne, on définit par induction la catégorie abélienne des \emph{n-hypercomplexes} de la façon suivante:
\begin{itemize}
 \item Les 1-hypercomplexes sont les complexes d'objets de $\mathcal{C}$.
 \item Les n-hypercomplexes sont les complexes de (n-1)-hypercomplexes.
\end{itemize}
On notera $\boldsymbol{C}^n(\mathcal{C})$ la catégorie abélienne des n-hypercomplexes d'objets de $\mathcal{C}$. Par exemple les 2-hypercomplexes sont les complexes doubles. Un n-hypercomplexe est donc la donnée pour tout ${\boldsymbol{k}}\in \mathbb{Z}^n$ d'un objet 
$X^{\boldsymbol{k}}$ de $\mathcal{C}$ et, pour tout $1\leq i\leq n$ de morphismes $d^{(i){\boldsymbol{k}}}:X^{\boldsymbol{k}}\to X^{{\boldsymbol{k}}+\boldsymbol{1}_i}$ vérifiant les propriétés suivantes:
\[\begin{array}{lllccc}
       d^{(i)}\circ d^{(i)}=0 & \text{pour tout} & i \\
       d^{(i)}\circ d^{(j)}=d^{(j)}\circ d^{(i)} & \text{pour tout} & (i,j)
   \end{array}\]
pour les exposants ${\boldsymbol{k}}$ convenables.
\end{de}

Soit $X$ un $n$-hypercomplexe, pour tout $1\leq i\leq n$ et tout $m\in\mathbb{Z}$ on note $X_i^m$ le $(n-1)$-hypercomplexe composé des $X^{\boldsymbol{k}}$ avec $k_i=m$ et des
différentielles correspondantes. Les différentielles $d^{(i){\boldsymbol{k}}}$ avec $k_i=m$ définissent un morphisme:
\[d_i^m:X_i^m\to X_i^{m+1}\]
qui vérifie $d_i^{m+1}\circ d_i^m=0$ par définition d'un $n$-hypercomplexe. On a donc pour tout $1\leq i\leq n$ un foncteur:
\[\begin{array}{lllccc}
   F_i: & \boldsymbol{C}^n(\mathcal{C}) & \to & \boldsymbol{C}(\boldsymbol{C}^{(n-1)}(\mathcal{C}))\\
        &  X & \mapsto & \{X_i^m,d_i^m\}_{m\in\mathbb{Z}}
  \end{array}
\]
de la catégorie des $n$-hypercomplexes dans la catégorie des complexes de $(n-1)$-hypercomplexes. On introduit alors le $(n-1)$-hypercomplexe:
\[H_i^p(X):=H^p(F_i(X)),\]
et le $n$-hypercomplexe:
\[H_i(X):=...\to H_i^p(X) \xrightarrow{0} H_i^{p+1}(X) \to ...\]
où toutes les flèches horizontales sont nulles.

\begin{de}\label{simple}
  Si un $n$-hypercomplexe $X$ vérifie la propriété de finitude suivante:
  \begin{equation}\label{finitude}
   \text{pour tout}~ m\in\mathbb{Z}~ \text{l'ensemble} 
   ~\{(k_1,...,k_n)\in\mathbb{Z}^n\mid k_1+...+k_n=m,X^{{\boldsymbol{k}}}\neq0\}~\text{est fini},
  \end{equation}
alors on peut associer à $X$ un complexe simple $s(X)$. On pose 
\[s(X)^m:=\bigoplus_{k_1+...+k_n=m}X^{\boldsymbol{k}}.\]
   Soit ${\boldsymbol{k}}\in\mathbb{Z}^n$ tel que $k_1+...+k_n=m$. On note $i_{{\boldsymbol{k}}}:X^{{\boldsymbol{k}}}\to s(X)^m$ et $p_{{\boldsymbol{k}}}:s(X)^m\to X^{\boldsymbol{k}}$ les morphismes naturels. On peut alors 
   définir la
   différentielle $d^m_{s(X)}:s(X)^m\to s(X)^{m+1}$ du complexe $s(X)$ par:
   
   \[p_{\boldsymbol{l}}\circ d^m_{s(X)}\circ i_{{\boldsymbol{k}}}=\left\{
   \begin{array}{ccclll}
    (-1)^{k_1+...+k_{j-1}}d^{(j){\boldsymbol{k}}} & \text{si} & \#\{i\mid k_i\neq l_i\}=1 & \text{où $j$ vérifie $k_j\neq l_j$}\\
    0 & \text{sinon}
   \end{array}\right.   
   \]
pour tout ${\boldsymbol{k}}$ et $\boldsymbol{l}$ vérifiant $k_1+...+k_n=m$ et $l_1+...+l_n=m+1$. On peut alors vérifier que $d^{m+1}_{s(X)}\circ d^m_{s(X)}=0$ et $(s(X),d_{s(X)})$ est 
donc bien un complexe. On a défini un foncteur 
\[\begin{array}{lllccc}
   s: & \boldsymbol{C}_f^n(\mathcal{C}) & \to & \boldsymbol{C}(\mathcal{C})\\
        &  X & \mapsto & (s(X),d_{s(X)})
  \end{array}
\]
où $\boldsymbol{C}_f^n(\mathcal{C})$ est la catégorie des $n$-hypercomplexes vérifiant la propriété \eqref{finitude}. De plus on observe facilement que $s(.)$ est un foncteur exact.
   \end{de}

\begin{theorem}\label{quasihyp}
 
 Soit $f:X\to Y$ un morphisme de $n$-hypercomplexes où X et Y vérifient la propriété \eqref{finitude} et supposons que $f$ induise un isomorphisme:
 \[f:H_1(H_2(...H_n(X)...))\simeq H_1(H_2(...H_n(Y)...)).\]
Alors $s(f):s(X)\to s(Y)$ est un quasi-isomorphisme.
 
\end{theorem}

\begin{proof}
 
 On raisonne par récurrence sur l'entier $n$. Pour $n=1$ c'est la définition d'un quasi-isomorphisme, pour $n=2$ c'est le théorème 1.9.3 de \cite{kash}. On suppose que $n\geq 3$.
 Pour tout $p\in \mathbb{Z}$, on a deux $(n-1)$-hypercomplexes, $H_n^p(X)$ et $H_n^p(Y)$, qui vérifient les hypothèses du théorème et donc par hypothèse de récurrence 
 $f$ induit un quasi-isomorphisme entre $s\left(H_n^p(X)\right)$ et $s\left(H_n^p(Y)\right)$. Or $H_n^p(X)=H^p(F_n(X))$ et $H^p(.)$ est un foncteur additif, il commute donc avec le foncteur
 $s(.)$ et $f$ induit un quasi-isomorphisme entre $H^p(\{s(X_n^m),s(d_n^m)\}_{m\in\mathbb{Z}})$ et $H^p(\{s(Y_n^m),s(d_n^m)\}_{m\in\mathbb{Z}})$ pour tout $p\in\mathbb{Z}$. Mais ce 
 quasi-isomorphisme correspond aux conditions du théorème pour les complexes doubles $\{s(X_i^m),s(d_i^m)\}_{m\in\mathbb{Z}}$ et $\{s(Y_i^m),s(d_i^m)\}_{m\in\mathbb{Z}}$,
 les complexes simples associés à ces deux complexes doubles sont donc quasi-isomorphes par hypothèse de récurrence pour $n=2$. En appliquant la définition du foncteur $s$ on montre alors que
 ces deux derniers complexes simples sont en fait les complexes simples associés à $X$ et à $Y$ ce qui conclut la démonstration du théorème.

\end{proof}

\begin{corollary}\label{quasinul}

Soit $X$ un $n$-hypercomplexe tel qu'il existe un indice $i$ pour lequel le complexe $F_i(X)$ soit exact, alors $s(X)$ est quasi-isomorphe au complexe nul.

\end{corollary}

\begin{proof}

Le théorème précédent est évidemment vérifié si l'on permute les indices des $H_i$. Si le complexe $F_i(X)$ est exact alors $H_i(X)\simeq H_i(0_n)$ où $0_n$ est le $n$-hypercomplexe nul. 
On a donc
 \[H_1(...H_{i-1}(H_{i+1}(...H_n(H_i(X))...)\simeq H_1(...H_{i-1}(H_{i+1}(...H_n(H_i(0_n))...)\]
et on peut appliquer le théorème précédent, $s(X)\simeq s(0_n)$, $s(X)$ est quasi-isomorphe au complexe nul.

\end{proof}

\begin{de}\label{cube}

Soit $\{X^{{\boldsymbol{k}}},f^{(i){\boldsymbol{k}}}\}_{{\boldsymbol{k}}\in\mathbb{Z}^n, 1\leq i\leq n}$ une famille d'objets de $\mathcal{C}$ et de morphismes $f^{(i){\boldsymbol{k}}}:X^{{\boldsymbol{k}}}\to X^{\boldsymbol{k+1}_i}$,
  on appelle \emph{hypercube associé à $X$} le $n$-hypercomplexe noté $\textrm{Cube}(X)^\bullet$ vérifiant 
\[
\textup{Cube}(X)^{k_1,...,k_n} =
\left\{\begin{array}{lllccc} X^{k_1-1,...,k_n-1}  & \text{si} & {\boldsymbol{k}}\in\{0,1\}^n\\
   0 & \text{sinon} 
\end{array}\right.\] 
les morphismes étant ceux donnés par les $f^{(i){\boldsymbol{k}}}$. On vérifie facilement que $\textup{Cube}(.)$ définit un foncteur exact.
\end{de}

Par exemple, pour $n=3$ on a 
\[\xymatrix{
 && X^{-1,0,0} \ar[rr] && X^{0,0,0}\\
&X^{-1,-1,0} \ar[ru] \ar[rr] && X^{0,-1,0} \ar[ru] \\
\textup{Cube}(X)= &&&&\\
&& X^{-1,0,-1}  \ar[uuu] \ar[rr] &&  X^{0,0,-1} \ar[uuu] \\
&X^{-1,-1,-1} \ar[rr] \ar[ru] \ar[uuu] && X^{0,-1,-1} \ar[ru] \ar[uuu]
}
\]

où le reste de l'hypercomplexe est nul et $X^{-1,-1,-1}$ est en degré $(0,0,0)$.

\section{Filtrations compatibles}\label{appB}

Les définitions qui suivent ont été introduites par Morihiko Saito dans \cite{HM1}

\begin{de}

Soit $A$ un objet de la catégorie abélienne $\mathcal{C}$ et $A_1,...,A_n\subseteq A$ des sous-objets de $A$. On dit que $A_1,...,A_n$ sont des \emph{sous-objets compatibles} de $A$ si il existe un $n$-hypercomplexe $X$ satisfaisant à:
\begin{enumerate}
\item $X^{\boldsymbol{k}}=0$ si $\boldsymbol{k}\not\in\{-1,0,1\}^n$.
\item $X^\mathbf{0}=A$.
\item $X^{\mathbf{0}-\mathbf{1}_i}=A_i$ pour $1\leq i\leq n$.
\item Pour tout $1\leq i\leq n$ et tout $\boldsymbol{k}\in\{-1,0,1\}^n$ tel que $k_i=0$, la suite
\[0\to X^{\boldsymbol{k-1}_i}\to X^{\boldsymbol{k}}\to X^{\boldsymbol{k+1}_i}\to 0\]
est une suite exacte courte.
\end{enumerate}

\end{de}

\begin{Rem}\label{remcompatibilité}
\begin{itemize}
\item En utilisant les propriétés universelles fournies par les suites exactes courtes on observe que si les sous-objets $A_1,...,A_n$ sont compatibles, alors le $n$-hypercomplexe $X$ est déterminé de manière unique. Par exemple si $\boldsymbol{k}\in \{-1,0\}^n$ et si $I=\{i;k_i=-1\}\subset\{1,...,n\}$ alors \[X^{\boldsymbol{k}}=\bigcap_{i\in I}A_i.\]
\item Si $n=1$, le complexe $X$ est la suite exacte courte
\[0\to A_1\to A\to A/A_1\to 0.\]
\item Si $n=2$ deux sous-objets $A_1$ et $A_2$ sont toujours compatibles et $X$ est le complexe double suivant
\[\xymatrix{A_1/(A_1\cap A_2) \ar[r]  & A/A_2 \ar[r] & A/(A_1+A_2) \\
A_1 \ar[r]\ar[u] & A \ar[r]\ar[u] & A/A_1\ar[u]\\
A_1\cap A_2 \ar[r]\ar[u] & A_2 \ar[r]\ar[u] & A_2/(A_1\cap A_2).\ar[u]
}\]
\item Si $n\geq 3$ des sous-objets $A_1,...,A_n$ ne sont pas compatibles en général.
\item Par définition si $A_1,...,A_n\subseteq A$ sont compatibles alors pour tout $I\subset\{1,...,n\}$ les sous-objets $(A_i)_{i\in I}\subseteq A$ sont compatibles et l'hypercomplexe correspondant est le $\#I$-hypercomplexe $X_I$ dont les objets sont les $X^{\boldsymbol{k}}$ tels que $k_i=0$ pour tout $i\in I^c$.
\end{itemize}
\end{Rem}

\begin{de}\label{Fcomp}

Soient $F_{\bullet}^1,...,F_\bullet^n$ des filtrations croissantes indexées par $\mathbb{Z}$ d'un objet $A$, on dit que ces filtrations sont \emph{compatibles} si pour tout $\boldsymbol{\ell}\in\mathbb{Z}^n$ les sous-objets $F_{\ell_1}^1,...,F_{\ell_n}^n$ de $A$ sont compatibles.

\end{de}

\begin{Rem}

\begin{itemize}

\item D'après la remarque précédente toute sous famille d'une famille de filtrations compatibles est compatible.
\item On peut montrer que si $F_{\bullet}^1,...,F_\bullet^n$ sont compatibles alors pour tout $\ell\in\mathbb{Z}$ les filtrations induites par $F_\bullet^1,...,F_\bullet^{n-1}$ sur $\textup{gr}_\ell^{F_n}$ sont compatibles.
\item Si $F_{\bullet}^1,...,F_\bullet^n$ sont compatibles alors les filtrations induites sur $F_{\ell_1}^1\cap...\cap F_{\ell_n}^n$ sont compatibles.
\end{itemize}

\end{Rem}

La proposition suivante correspond à \cite[corollaire 1.2.13]{HM1}

\begin{proposition}\label{commutmultgrad}

Soit $F_{\bullet}^1,...,F_\bullet^n$ des filtrations compatibles d'un objet $A$. L'objet obtenu en appliquant successivement les gradués $\textup{gr}_{\ell_{\sigma(j)}}^{F_{\sigma(j)}}$ 
par rapport aux filtrations $F_{\sigma(j)}$ induites sur $\textup{gr}_{\ell_{\sigma(j-1)}}^{F_{\sigma(j-1)}}...\textup{gr}_{\ell_{\sigma(1)}}^{F_{\sigma(1)}}A$ pour $1\leq j\leq n$ ne dépend pas de la permutation $\sigma$ de $\{1,...,n\}$ et est égal à
\[\frac{F_{\ell_1}^1A\cap...\cap F_{\ell_n}^nA}{\sum_j F_{\ell_1}^1A\cap...\cap F_{\ell_{j}-1}^1A\cap...\cap F_{\ell_n}^nA}.\]

\end{proposition}

\bibliographystyle{smfalpha}
\bibliography{biblio}

\end{document}